\documentclass[10pt,a4paper]{article}
\usepackage{amssymb}

\usepackage{amscd, amsrefs,amsmath,amsthm,amsfonts,esint}
\usepackage{mathtools}
\usepackage{hyperref}

\newtheorem{theorem}{Theorem}[section]
\newtheorem{lemma}{Lemma}[section]

\newtheorem{corollary}{Corollary}[section]
\newtheorem{proposition}{Proposition}[section]

\theoremstyle{definition}
\newtheorem{definition}[theorem]{Definition}

\theoremstyle{remark}

\newcommand{\ds}{\displaystyle}

\newcommand{\Z}{\mathbb{Z}}
\newcommand{\R}{\mathbb{R}}
\newcommand{\Rn}{{\R^n}}
\newcommand{\C}{\mathbb{C}}

\newcommand{\cA}{\mathcal{A}}
\newcommand{\cF}{\mathcal{F}}
\newcommand{\cH}{\mathcal{H}}
\newcommand{\cHloc}{{\mathcal{H}_{\text{loc}}}}
\newcommand{\cM}{\mathcal{M}}
\newcommand{\cR}{\mathcal{R}}
\newcommand{\cS}{\mathcal{S}}
\newcommand{\cRjeta}{{\cR_{j,\eta}}}
\newcommand{\cRjpsi}{{\cR_j^\psi}} 

\newcommand{\bmo}{{\rm bmo}}
\newcommand{\BMO}{{\rm BMO}}
\newcommand{\vmo}{{\rm vmo}}
\newcommand{\VMO}{{\rm VMO}}
\newcommand{\CMO}{{\rm CMO}}
\newcommand{\cmo}{{\rm cmo}}
\newcommand{\lmo}{{\rm lmo}}
\newcommand{\LMO}{{\rm LMO}}
\newcommand{\BLO}{{\rm BLO}}
\newcommand{\bmop}{{\text{bmo}, p}}
\newcommand{\BMOloc}{{{\rm BMO}_{\text{loc}}}}
\newcommand{\LMOloc}{{{\rm LMO}_{\text{loc}}}}

\newcommand{\hone}{{h^1}}
\newcommand{\Hone}{{H^1}}
\newcommand{\hp}{{h^p}}
\newcommand{\Hp}{{H^p}}
\newcommand{\Honeb}{{H^1_b}}
\newcommand{\cHoneb}{{\cH^1_b}}
\newcommand{\honeb}{{h^1_b}}

\newcommand{\honeperb}{{h^1_{{\rm Perez},b}}}
\newcommand{\honeatomb}{{h^1_{{\rm atom},b}}}
\newcommand{\honefinb}{{h^1_{{\rm finite},b}}}

\newcommand{\Lone}{{L^1}}
\newcommand{\Loneloc}{{L^1_{{\rm loc}}}}
\newcommand{\Lp}{{L^p}}
\newcommand{\Linfty}{{L^\infty}}
\newcommand{\Linftyc}{{L^\infty_c}}

\newcommand{\Cinftyc}{{C^\infty_c}}

\newcommand{\Mpsi}{{M_\psi}}
\newcommand{\fM}{\mathfrak{M}}
\newcommand{\fMnt}{{\fM_{\text{nt}}}}
\newcommand{\cMb}{{\cM_b}}

\newcommand{\fhat}{{\hat{f}}}

\newcommand{\lbk}{\left(}
\newcommand{\rbk}{\right)}
\newcommand{\lsk}{\left[}
\newcommand{\rsk}{\right]}
\newcommand{\lb}{\lbrace}
\newcommand{\rb}{\rbrace}

\newcommand{\ra}{\rightarrow}

\DeclareMathOperator*{\supp}{supp}
\DeclareMathOperator*{\Lip}{Lip}
\DeclareMathOperator*{\Log}{Log}
\DeclareMathOperator*{\grad}{\!\nabla\!}
\DeclareMathOperator*{\sgn}{sgn}

\numberwithin{equation}{section}

\title{$\hone$ boundedness of Localized Operators and Commutators with bmo and lmo}

\author{Galia Dafni, Chun Ho Lau \thanks{The authors were partially supported by the Natural Sciences and Engineering Research Council
(NSERC) of Canada, and the Centre de recherches math\'{e}matiques (CRM)} }

\date{}

\begin{document}

\maketitle





\begin{abstract}
We first consider two types of localizations of singular integral operators of convolution type, and show, under mild decay and smoothness conditions on the auxiliary functions, that their boundedness on the local Hardy space $\hone(\Rn)$ is equivalent. We then study the boundedness on $\hone(\Rn)$ of the commutator $[b,T]$ of an inhomogeneous singular integral operator with $b$ in $\bmo(\Rn)$, the nonhomogeneous space of functions of bounded mean oscillation. We define local analogues of the atomic space $\Honeb(\Rn)$ introduced by P\'erez in the case of the homogeneous Hardy space and $\BMO$, including a variation involving atoms with approximate cancellation conditions.  For such an atom $a$, we prove integrability of the associated commutator maximal function and of $[b,T](a)$.  For $b$ in $\lmo(\R^n)$, this gives $\hone$ to $L^1$ boundedness of $[b,T]$. Finally,  under additional approximate cancellation conditions on $T$, we show boundedness to $\hone$. 
\end{abstract}


\bibliographystyle{amsplain}
\section{Introduction} 
The real Hardy spaces $\Hp(\Rn)$ are a class of function spaces that has been extensively studied in harmonic analysis for over 50 years.  In particular, while the boundedness of singular integral operators on $\Lp(\Rn)$ for $1 < p < \infty$ does not extend to $0 < p \leq 1$, under sufficient smoothness and cancellation conditions such operators are bounded on $\Hp(\Rn)$, $0 < p \leq 1$.  At the other end of the scale, $p = \infty$, the space $\BMO(\Rn)$ of functions of bounded mean oscillation introduced by John and Nirenberg \cite{JN} replaces $\Linfty$ and is the dual of $\Hone$ - see the fundamental paper of Fefferman and Stein \cite{FS} and more recent exposition in \cite{SteinHA}.

The focus of our paper is on a nonhomogeneous version of these spaces, introduced by Goldberg  \cite{Goldberg1, Goldberg2} under the name ``local Hardy spaces", denoted $\hp(\Rn)$. Membership in $\hp(\Rn)$ does not require the global vanishing moment conditions necessary for $f \in \Hp(\Rn)$ and so, unlike  the case of $\Hp(\Rn)$, important classes of smooth functions such as the Schwartz spaces $\cS(\Rn)$ are contained (and are dense) in $\hp(\Rn)$.  Moreover,  $\hp(\Rn)$ is closed under multiplication by smooth cut-off functions, which is why these are sometimes called ``localizable" Hardy spaces, and are more suitable to analysis on different settings such as domains, manifolds and spaces of homogeneous type, as well as the study of partial differential equations and pseudo-differential operators.

We restrict ourselves to $\hone(\Rn)$, which satisfies $\Hone(\Rn) \subsetneq \hone(\Rn) \subsetneq \Lone(\Rn)$, and the nonhomogeneous version of $\BMO(\Rn)$, denoted $\bmo(\Rn)$ and identified by Goldberg as the dual of $\hone(\Rn)$.  In addition to characterizations via maximal functions and atomic decomposition, Goldberg also identified $\hone(\Rn)$  as the space of functions in $\Lone(\Rn)$ whose ``local Riesz transforms" $r_j$, $j = 1, \ldots, n$, are in $\Lone(\Rn)$.  The $r_j$ are obtained by localizing the multipliers of the Riesz transforms $R_j$ on the Fourier transform side via multiplication with a Schwartz function vanishing near the origin.  We consider more general localizations of Calder\'on--Zygmund convolution-type singular integral operators, obtained either by multiplying the kernel by an appropriate function (as in the ``truncated" Hilbert transform defined in \cite{Goldberg1}), or localizing on the Fourier transform side.  We show, under mild decay and smoothness assumptions on the auxiliary functions, that the boundedness of different localizations of the same kernel on $\hone(\Rn)$ is equivalent (Theorem~\ref{Mainthm1}), allowing us to characterize $\hone$ by a broader class of localized Riesz transforms (Corollary~\ref{localriesz}).

We then turn to another important class of operators, the commutators of singular integral operators with multiplication operators: $[b,T]f:= bT(f) - T(bf)$.  Coifman, Rochberg and Weiss \cite{CRW} proved that $[b,T]$, where $b\in BMO(\Rn)$ and $T$ is a Calder\'on--Zygmund operator, is bounded on $\Lp(\Rn)$ for all $1<p<\infty$. They also proved a converse: for the Riesz transforms $R_j$, if $[b,R_j]$, $j = 1, \ldots, n$, are all bounded on $L^p(\Rn)$ for some $1<p<\infty$, then $b\in BMO(\Rn)$. 
Subsequently, Uchiyama \cite{Uchiyama1} proved that $[b,T]$ is compact on $L^p(\Rn)$, $1<p<\infty$, if and only if $b\in \CMO(\Rn)$, where $\CMO(\Rn)$ is the $\BMO(\Rn)$-closure of $\Cinftyc(\Rn)$. Janson \cite{Janson1} extended the boundedness of the commutator to Orlicz-type spaces. Recently, Hyt\"{o}nen \cite{Hytonen1} completely characterized $\Lp \mapsto L^q$ boundedness.

The above results are all for $L^p$ spaces with $1<p<\infty$. We are interested in the endpoint case $p = 1$. For a Calder\'on--Zygmund operator $T$, we have weak-type $(1,1)$ boundedness as well as from $\Hone(\Rn)$ to $\Lone(\Rn)$. Does the commutator $[b,T]$ also satisfy this type of boundedness? The answer turns out to be no. Harboure, Segovia and Torrea \cite{HST} proved that there is no way to get boundedness of $[b,H]$ from $\Hone(\R)$ to $\Lone(\R)$, where $H$ be the Hilbert transform, unless $b$ is a constant function, and this also holds in the case of boundedness of $[b,H]$ from $\Linftyc(\R)$ to $BMO(\R)$.

This means we need to seek a smaller subspace of $\Hone$ to get the integrability of the commutator.  In addition to proving $L\Log L$ to weak-$\Lone$ boundedness of the commutator $[b,T]$, P\'erez \cite{Perez} introduced a special type of $\Hone$ atom that has extra cancellation against $b$, and showed that the commutator $[b,T](a)\in \Lone(\Rn)$ for all such $a$.  He then defined the space $\Honeb(\Rn)$ via atomic decompositions.  The boundedness on infinite linear combinations of atoms, however, does not follow from the boundedness on one atom (see \cite{Bownik}), so another characterization of the space was needed.  

Ky \cite{Ky} introduced another space $\Honeb(\Rn)$ through a specially defined commutator maximal function, and proved that it is the largest subspace of $\Hone(\Rn)$ such that the commutator $[b,T]$ is bounded from this space to $\Lone(\Rn)$.  As P\'erez's atomic space is continuously contained in Ky's space, the boundedness on atoms could be extended to get boundedness on the whole atomic space.  The reverse inclusion, namely whether every function in Ky's maximal $\Honeb(\Rn)$ space has an atomic decomposition into P\'erez atoms, was not shown.  Ky also showed that under suitable assumptions, $[b,T]$ is  bounded from $\Hone$ to $\hone(\Rn)$ and also to $\Hone(\Rn)$, and these results were extended  to metric measure space of homogeneous type by Fu, D. Yang and S. Yang \cite{FYY1}.
 
Unlike the case for $\Hone(\Rn)$, commutators of the form $[b,T]$ with nonconstant $b$ can be bounded on $\hone(\Rn)$.  Hung and Ky \cite{HungKy} proved such boundedness for $b$ with logarithmically vanishing mean oscillation and $T$ in a class of pseudo-differential operators.   

In this paper we introduce an $\hone(\Rn)$ version of the P\'erez atomic space, as well as another atomic space which we call $\honeatomb(\Rn)$, containing it, defined using new $\honeb$ atoms which only satisfy approximate cancellation conditions.  We then study the boundedness of $[b,T]$ for $b \in \bmo(\Rn)$ and $T$ an inhomogeneous singular integral operator, and show that such commutators take $\honeb$ atoms to $\Lone(\Rn)$ (Theorem~\ref{h1bL1}).  Moreover, imposing nonhomogeneous cancellation conditions on $T$ by considering $T^*1$ and $T^*b$, we are able to show boundedness to $\hone(\Rn)$ (Theorem~\ref{thm: h1bh1}).  However, as in the case of the P\'erez atomic space $\Honeb$, it is not possible to extend this boundedness to all of $\honeatomb(\Rn)$ without a different characterization of the space.  While we do not have a maximal characterization, we do show that a local analogue of Ky's commutator maximal function takes $\honeb$ atoms to $\Lone(\Rn)$.
 
Finally, restricting $b$ to $\lmo(\Rn)$, the functions of logarithmically vanishing mean oscillation in $\bmo(\Rn)$, we are able to identify the space $\honeatomb(\Rn)$ with $\hone(\Rn)$ (Theorem~\ref{h1atombish1}) and thus conclude the boundedness of the commutator on the whole space, generalizing the results of Hung and Ky \cite{HungKy}.

Section~\ref{sec:prelim} gives some background on $\hone$, with particular emphasis on atomic decompositions and approximate cancellation conditions, as well as on $\bmo$ and $\lmo$.  Section~\ref{sec:LCK} deals with the various localizations of  Calder\'on--Zygmund operators of convolution type and their boundedness on $\hone$, and ends with the definition of inhomogeneous singular integral operators.  Section~\ref{sec:spaces} introduces the commutator maximal functions and atomic spaces in the nonhomogeneous case, and proves relations between them.  Finally, Section~\ref{sec:commutators} contains the results on the boundedness of the commutators on the atoms.

\section{Preliminaries}
\label{sec:prelim}
The Lebesgue measure of a set $A\subset \Rn$ is denoted by $|A|$, and  the average of an integrable function $f$ on a set $E$ of positive finite measure by
$$f_E : = \fint_E f : = \frac{1}{|E|} \int_E f(x)dx.$$

Balls will be denoted by $B(x_0, r)$ and the radius of a ball $B$ by $r(B)$. The dilated ball $B(x_0, \delta r)$, $\delta > 0$, will be denoted by $\delta B$.

The constant $C$ used in inequalities will may vary from line to line.  
We write $\alpha \lesssim \beta$ if there exists a constant $C$ (which may depend on other factors but not on $\alpha$ and $\beta$) such that $\alpha \leq C\beta$, and $\alpha \approx \beta$ if $\alpha \lesssim \beta$ and $\beta \lesssim \alpha$.

We denote the Schwartz space by $\cS(\Rn)$ and the seminorms defining its topology by 
$$\|\varphi\|_{\alpha,\beta} := \sup_{x \in \Rn} |x^\alpha \partial_x^\beta \varphi(x)|, \quad \alpha, \beta \in \Z_+^n.$$
Its dual, $\cS'(\Rn)$, is the space of tempered distributions.  

\subsection{Local Hardy spaces}
Goldberg \cite{Goldberg1,Goldberg2} introduced nonhomogeneous versions of the real Hardy spaces $H^p(\Rn)$, called ``local Hardy spaces" and denoted by $h^p(\Rn)$, and characterized them by means of maximal functions and atomic decomposition.  
\begin{definition}[\cite{Goldberg1}]
\label{def:Goldberg-atom}
For  $0 < p \leq 1$, a (normalized) Goldberg $h^p$ atom $a$ is a function supported in a cube $Q$ with 
\begin{itemize}
\item[(i)] $\|a\|_\Linfty \leq |Q|^{-1}$; 
\item[(ii)] if $|Q| < 1$, $\int a(x) x^\alpha dx = 0$
for $|\alpha| \leq \lfloor n(1/p-1) \rfloor$.  
\end{itemize}
\end{definition}

An atomic decomposition for $h^p$ where the atoms need only have approximate moment conditions was introduced in \cite{DafniThesis} and later refined in \cite{DY,DMY} for the case $p =1$ in the more general setting of spaces of homogeneous type  (see also \cite{Komori, DLPV1, DLPV2} for $p < 1$).  
\begin{definition}[\cite{DY}]
\label{def:R-atom}
Fix $R > 0$, $1 < q \leq \infty$. We say $a$ is an  $R$-approximate $(1,q)$ atom if 
\begin{itemize}
\item $\supp(a)\subset B(x_0,r)$ for some $x_0\in \Rn$ and $r>0$;
\item $\|a\|_{L^q}\leq |B(x_0,r)|^{\frac1{q} - 1}$; 
\item $|\int a|\leq [\log(1+\frac R r)]^{-1}.$
\end{itemize}
\end{definition}

Any of the following equivalent characterizations can be taken as the definition of $\hone(\Rn)$.
\begin{proposition}
\label{prop:h^1}
Let $f \in \cS'(\Rn)$.  Then the following are equivalent:
\begin{enumerate}
\item For a fixed $\psi \in \cS(\Rn)$ with $\int \psi \neq 0$, 
$$\|f\|_{\hone,\psi}:= \|\Mpsi(f)\|_{L^1(\Rn)} < \infty,$$
where
$$\Mpsi(f)(x) = \sup_{0 < t < T} |f * \psi_t(x)|, \quad \psi_t(y) = t^{-n}\psi(t^{-1}y).$$
\item For a fixed $T > 0$,
$$\|f\|_{\hone,T}:= \|\cM(f)\|_{L^1(\Rn)} < \infty,$$
where
$$\cM(f)(x):=\sup_{\phi} |\langle f, \phi \rangle|, $$
and the supremum is taken over all $\phi\in C^1(\Rn)$ with $\supp(\phi)\subset B(x,t)$, $0<t<T$, $\|\phi\|_{L^{\infty}}\leq |B(x,t)|^{-1}$ and $\|\nabla\phi\|_{L^{\infty}}\leq [t|B(x,t)|]^{-1}$.
\item For $N$ sufficient large
$$\|f\|_{\hone,\rm{max}}:= \|\fM(f)\|_{L^1(\Rn)} < \infty,$$
where
$$\fM f(x):=\sup\{ |f * \varphi_t(x)|: 0 < t < 1, \varphi\in \cF_N\}$$
and
$$\cF_N = \{\varphi\in \cS(\Rn): \|\varphi\|_{\alpha,\beta} \leq 1\; \forall\; |\alpha|, |\beta| \leq N\}.$$
\item 
$$\|f\|_{\hone,\rm{atom}}:= \inf \sum |\lambda_j| < \infty,$$
where the infimum is taken over all decompositions $f = \sum \lambda_j a_j$ in $\cS'(\Rn)$ with $\{\lambda_j\} \in \ell^1$ and where the $a_j$ are Goldberg $\hone$ atoms as in Definition~\ref{def:Goldberg-atom}.
\item For a fixed $R > 0$, $1 < q \leq \infty$,
$$\|f\|_{\hone,R}:= \inf \sum |\lambda_j| < \infty,$$
where the infimum is taken over all decompositions $f = \sum \lambda_j a_j$ in $\cS'(\Rn)$ with $\{\lambda_j\} \in \ell^1$ and where the $a_j$ are $R$-approximate $(1,q)$ atoms as in Definition~\ref{def:R-atom}.
\end{enumerate}
In case these conditions hold, we have
$$\|f\|_{\hone,\psi} \approx \|f\|_{\hone,T} \approx \|f\|_{\hone,\rm{max}} \approx \|f\|_{\hone,\rm{atom}} \approx \|f\|_{\hone,R},$$
where the constants depend on the choices of $T$, $R$ and $N$.
\end{proposition}

The equivalence of conditions (1), (3) and (4) of Proposition~\ref{prop:h^1} is contained in \cite[Theorem 1 and Lemma 5]{Goldberg1}.  The equivalence of conditions (2), (4) and (5) is shown in \cite[Section 7]{DY} on a doubling metric-measure space, with $C^1$ replaced by Lipschitz, for arbitrary $T$, $R$ and $q$ (see also \cite[Lemma 1]{DLPV1} for a proof that (4) implies (2) in the more general case of all $p \leq 1$).  The maximal functions $\cM$ and $\fM$ are often given the name {\em grand maximal function}.  It is also possible to define nontangential versions of these (see Definition~\ref{def:Ky} below for the case of $\Hone$) and get an equivalent characterization of $\hone$ (resp.\ $\Hone$) - this is part of Goldberg's theorem (see \cite[Section III.1]{SteinHA} for the homogeneous Hardy spaces).  Further characterizations of local Hardy spaces on spaces of homogeneous type were given in \cite{HYY}.

In what follows, we will use mostly condition (2) with $T= 1$, and abbreviate $\|f\|_{\hone,1}$ to $\|f\|_{\hone}$, and condition (5) with $R = 1$ and $q = 2$, in which case we will call the $1$ approximate $(1,2)$ atoms simply $\hone$ atoms.  

We will also need the notion of a molecule, which is a generalization of an atom not requiring compact support.  Molecules are useful in showing boundedness of operators on $h^p$ since, typically, singular integral operators take atoms to molecules.  If we know that molecules are in $h^p$ with bounded norm, then we have shown the boundedness of the operator on atoms.

\begin{definition} [\cite{DLPV1}]
\label{def:molecule}
Let $1 < s< \infty$ , $\lambda > n ({s}-1)$. A measurable function $M$ on $\Rn$ is called a 
		$(s,\lambda)$ molecule (for $\hone(\Rn)$) if there exists a ball $B = B(x_0,r)$  such that
		\begin{itemize}	      
		\item[\textnormal{M1.}] $\|M\|_{L^s(B)} \leq \, r^{n\left(\frac{1}{s}-1\right)}$
			\item[\textnormal{M2.}] $\ds \| M \, |\cdot - x_0|^{\frac{\lambda}{s}} \,  \|_{L^{s}(B^c)} \leq \, r^{\frac{\lambda}s + n \left(\frac{1}{s}-1\right)}$
			\item[\textnormal{M3.}] $\ds  \left| \int M \right| \leq [\ln(1+r^{-1})]^{-1}.$
		\end{itemize}
	\end{definition}

\begin{proposition}[\cite{DLPV1}]
There exists a constant $C_{n,s,\lambda}$ such that $\|M\|_\hone \leq C_{n,s,\lambda}$ every $(s,\lambda)$ molecule $M$.  

Conversely, for any $f\in \hone(\R^n)$, there exist a sequence $\lb \lambda_j\rb\in \ell^1$ and a sequence of $h^1$ molecules $\lb M_j\rb$ such that 
$$f=\sum_{j=1}^{\infty} \lambda_j M_j \quad \text{in }\hone(\R^n).$$
Moreover,
$$\|f\|_{\hone} \approx \inf \sum |\lambda_j| < \infty,$$
where the infimum is taken over all such decompositions.
\end{proposition}

The following proposition (case $p = 1$ of \cite[Proposition 1]{DLPV2}) shows that the cancellation condition on atoms in Definition~\ref{def:R-atom} is not only sufficient but also necessary. 
\begin{proposition}[\cite{DLPV2}]
 \label{prop:meanestimate}
Suppose $g\in \hone(\Rn)$ with $\supp(g)\subset B(x_0,r)$. Then
$$\bigg|\int g \bigg| \lesssim \frac{\|g\|_{\hone}}{\log(1+r^{-1})}.$$
\end{proposition}

\subsection{Nonhomogeneous BMO}
Goldberg \cite[Corollary1]{Goldberg2} identifies the dual of $\hone(\Rn)$ with the space $\bmo(\Rn)$, a nonhomogeneous version of the John-Nirenberg space $\BMO(\Rn)$ of functions of bounded mean oscillation.  See Bourdaud \cite{Bourdaud} for an exposition of the properties of $\bmo(\Rn)$. In the following definition, we use the notation introduced in \cite{DY}, where it is shown that the definition is independent of the choice of scale used to distinguish large and small balls. Note that unlike $\BMO(\Rn)$, we do not need to take the space modulo constants in order to get a Banach space.
\begin{definition}
Let $b \in \Loneloc(\Rn)$. We say $b\in \bmo(\Rn)$ if 
\begin{align*}
\|b\|_\bmo:= \sup_{B}\fint_{B}|b(x)-c_B|dx <\infty,
\end{align*}
where the supremum is taken over all balls, and for each ball $B$ we set
\begin{equation}
\label{eqn: cB}
c_B:=\begin{dcases} b_B:= \fint_{B}b, & \quad \text{ if }r(B)<1, \\ 0,  & \quad \text{ if }r(B)\geq1. \end{dcases}
\end{equation}
\end{definition}
Since there will only be one $\bmo$ function under discussion at any point in time, we will continue to use the notation $c_B$ defined in \eqref{eqn: cB} in the rest of the paper.
We can then state a corresponding version of the John-Nirenberg inequality  for $b \in \bmo(\Rn)$  (see \cite[Theorem 3.1]{DY}):
 there exist $C,c>0$ such that for any ball $B$, 
$$
|\lb x\in B: |b(x)-c_B|>\lambda\rb |\leq C |B| e^{-\lambda c/\|b\|_{\bmo}}
$$
for all $\lambda>0$.
As consequence, one gets that  
\begin{equation}
\label{eqn:bmop}
\|b\|_\bmo \approx \|b\|_\bmop : =  \lbk \sup_{B} \fint_{B}|b(x)-c_B|^pdx \rbk^{1/p}
\end{equation}
for $1 \leq p < \infty$, with constants depending on $p$.  

This gives us the following useful inequality (see the proof of \cite[Lemma 2.4]{CRW} for the case $\delta = 1$ and $b \in \BMO(\Rn)$).
\begin{lemma}
Let $b\in \bmo(\Rn)$. Then for any $\delta>0$, $p\geq 1$, there exists $C_{n,p,\delta}>0$, independent  of $b$, $x_0$ and $r$, such that  
\begin{equation}
\label{eqn:bintegral}
r^{\delta}\int_{|x-x_0|>r} \frac{|b(x)-c_{B(x_0,r)}|^{p}}{|x-x_0|^{n+\delta}}\leq C_{n,p,\delta}\|b\|_{\bmo}^{p}.
\end{equation}
\end{lemma}

\begin{proof}
By translation invariance of $\bmo(\Rn)$, we may assume $x_0 = 0$.  Denote $B(0, 2^kr)$ by $B_k$, $k = 0, 1, 2,\ldots$.  Then using \eqref{eqn:bmop} and the fact that for $2^j r < 1$
\begin{equation}
\label{eqn:doubling} |c_{B_j} - c_{B_{j-1}}| =
 |b_{B_j} - b_{B_{j-1}}| \leq 2^n \|b\|_\bmo
\end{equation}
and $c_B \lesssim \frac{1}{\log(1+r^{-1})}$ (see  \cite[Lemma 6.1]{DY}), we have
\begin{align*}
r^{\delta}\int_{\Rn \setminus B_0} \frac{|b(x)-c_{B_0}|^{p}}{|x|^{n+\delta}}
& \lesssim \sum_{k = 0}^\infty 2^{-k\delta} |B_k|^{-1}\int_{B_{k+1} \setminus B_k}|b(x)-c_{B_{k+1}}|^{p} + |c_{B_{k+1}} - c_{B_0}|^p dx\\
& \lesssim\sum_{k = 0}^\infty 2^{-k\delta}  \|b\|^p_{\bmo} (1 + (\min\{k+1, \log_+ r^{-1}\})^p) dx\\
& \lesssim \|b\|_{\bmo}^p.
\end{align*}
\end{proof}

While $\bmo$ is the dual of the local Hardy space $\hone$, it does not make sense to call it ``local BMO" as it is actually strictly smaller.  It is more appropriate to give that name to the space consisting of locally integrable functions $b$ for which 
$$\Vert b\Vert_{\BMOloc(\Rn)} :=\sup_{r(B)< R}\fint_{B}|b(x)-b_B|dx<\infty,$$
where $R$ is some fixed constant (see \cite{TangL2}).  This space is strictly larger than $\BMO(\Rn)$ since it contains all Lipschitz functions, for example, and is independent of the choice of $R < \infty$.  Taking those elements of $\BMOloc$ for which $\displaystyle{\sup_{r(B) \geq R} \fint |b| < \infty}$ gives $\bmo(\Rn)$ (in \cite{Bourdaud} it is shown  that it suffices to take $\displaystyle{\sup_{r(B) = R} \fint |b| < \infty}$).  In particular $\BMOloc(\Rn) \cap L^p(\Rn) \subset \bmo(\Rn)$ for $1 \leq p < \infty$.

Vanishing mean oscillation refers to the property that
\begin{equation}
\label{eqn: VMO}
\lim_{R \ra 0}\sup_{r(B)< R}\fint_{B}|b(x)-b_B|dx = 0.
\end{equation}
For $b \in \BMO(\Rn)$, this holds if and only if $b \in \VMO(\Rn)$, the closure of the uniformly continuous functions in $\BMO(\Rn)$, as shown by Sarason \cite{Sarason1} for $n = 1$.  Analogously, we can denote by $\vmo(\Rn)$ the space of functions in $\bmo(\Rn)$ which have vanishing mean oscillation, or alternatively the closure of the bounded uniformly continuous functions in $\bmo(\Rn)$.  Finally, we use $\cmo(\Rn)$ to denote the closure of the continuous (or smooth) functions with compact support in $\bmo(\Rn)$.  More about these spaces can be found in \cite{Bourdaud} and \cite{dafnilocalvmo} (where the notation $\vmo$ is used for $\cmo$).  It was shown there that $\hone(\Rn)$ can be identified with the dual of $\cmo(\Rn)$ and that $b \in \cmo(\Rn)$ if and only if \eqref{eqn: VMO} holds and in addition
\begin{equation}
\label{eqn: cmo}
\lim_{R \ra \infty}\sup_{\substack{|B|\geq 1 \\ B\subset (B(0,R))^c  }} \fint_{B}|b| =0 .
\end{equation}

One can specify the rate at which the oscillation vanishes as the radius goes to zero in \eqref{eqn: VMO} (see \cite[Chapter 5]{Sarason1}, \cite{ShiTorchinsky}).  A special class are those functions of {\em logarithmic mean oscillation}, usually denoted $\LMO$ (not to be confused with {\em bounded lower oscillation} denoted by $\BLO$).  Bounded functions in this class have been identified as the pointwise multipliers of $H^1$ and $\BMO$ on the circle \cite{Stegenga}, on the sphere \cite{Li} and on spaces of homogeneous type \cite{ChangLi}.  The analogous result in the  case of $\bmo$ was shown in \cite{BF}.  These spaces have also proved useful in the study of PDE.  In \cite{TangL} a logarithmic mean oscillation condition is imposed on the coefficients of a parabolic equation, while in  \cite{Bernicot1, Bernicot2} a range of such conditions is imposed on the initial vorticity in the Euler and Navier-Stokes equations.  Note that the latter articles use slightly different notation (for example $\bmo$ there refers to ``local BMO" rather than Goldberg's nonhomogeneous space) and an $L^2$ oscillation, which is equivalent to the $L^1$ one by the John-Nirenberg inequality for these spaces (see \cite{ShiTorchinsky}).  In \cite{HungKy}, a space called $\LMO_\infty$ is introduced which consists of functions in $\BMOloc$ which have logarithmic vanishing mean oscillation on small balls while the mean oscillation on large balls can grow as a power of the radius.  This space is contained in  $\LMOloc(\Rn)$, defined as follows.

\begin{definition}
Let $b\in \Loneloc(\Rn)$. We say $b \in \LMOloc(\Rn)$ if 
$$\Vert b\Vert_{\LMOloc(\Rn)} := \sup_{r(B)<1}\frac{[\log(1+r(B)^{-1})]}{|B|}\int_{B}|b(x)-b_B|dx<\infty.$$
We also define $\lmo(\Rn):=\LMOloc(\Rn)\cap \bmo(\Rn)$ and
$$\Vert b\Vert_{\lmo(\Rn)} := \Vert b\Vert_{\LMOloc(\Rn)} + \sup_{r(B)\geq1}|b|_B<\infty.$$
\end{definition}

Since $[\log(1+r(B)^{-1})]^{-1} \ra 0$ as $r \ra 0$, we have that functions in $\LMOloc(\Rn)$ satisfy the vanishing mean oscillation condition \eqref{eqn: VMO} and therefore $\lmo(\Rn) \subset \vmo(\Rn)$.  If we take $b \in \LMOloc(\Rn) \cap L^p(\Rn)$, $1\leq p < \infty$, then $b$ will also satisfy condition \eqref{eqn: cmo} since  if $|B|\geq 1$ and  $B\subset \Rn \setminus B(0,R)\subset \Rn$, 
then
$\fint_B |b| \leq \|b\|_{L^p(B(0,R)^c)} \ra 0$ as $R \ra \infty$.  Thus $\LMOloc(\Rn) \cap L^p(\Rn) \subset \cmo(\Rn)$.

\section{Localized convolution kernels} 
\label{sec:LCK}
In addition to the characterizations via maximal functions and atomic decomposition, as in Proposition~\ref{prop:h^1}, Goldberg also characterized $\hone(\Rn)$  by singular integral operators in an analogous way to the characterization of $H^1(\Rn)$ via the Riesz transforms $R_j$, $j = 1, \ldots, n$, given by their multipliers as
$\widehat{R_j f}(\xi) = i\frac{\xi_j}{|\xi|}\fhat(\xi)$.

\begin{theorem}{\cite[Theorem 2]{Goldberg2}}
\label{thm:riesz}
Fix a function $\varphi \in \cS(\Rn)$ with $\varphi \equiv 1$ in a neighborhood of the origin.  Then
a distribution $f$ belongs to $\hone(\Rn)$ if and only if $f \in \Lone(\Rn)$ and $r_j f \in \Lone(\Rn)$, $j = 1, \ldots, n$, where
$r_j$ is defined by 
$$\widehat{r_j f}(\xi) = (1 - \varphi(\xi))i\frac{\xi_j}{|\xi|}\fhat(\xi).$$
\end{theorem}
See \cite{Pelsec} for a generalization of this result to the case $0 < p \leq 1$, using $\varphi$ of compact support.

For $n = 1$, in \cite{Goldberg1}, Goldberg also claims a characterization of $\hone(\R)$ via a localized version of the Hilbert transform.  This can also be described as a ``truncated" Hilbert transform (not to be confused with the usual truncation used to the defined the principal value integral).  We set
\begin{equation} 
\label{Hloc}
\cHloc(f)(x):= p.v. \int_{\R} \frac{\eta(x-y)f(y)}{x-y}dy,
\end{equation}
where $\eta \equiv 1$ on $(-1,1)$ and is supported in $[-2,2]$.  In \cite{DL} it was shown that for smooth $\eta$, the local Hilbert transform $\cHloc$ characterizes $\hone(\R)$.  

Note that the two localizations are different: $\cHloc$ is obtained by localizing the kernel of the Hilbert transform near $0$, while $r_j$ is obtained by localizing the Fourier transform of $R_j$ away from $0$.  This may cause some confusion.  For example, in order to characterize weighted local Hardy spaces, higher-dimensional analogues of the local Hilbert transforms, namely ``truncated" versions of the Riesz transforms where the kernel is multiplied by a smooth function of compacted support, were defined in \cite[Section 8]{TangL2} using the notation $R_j$.  However, the result referred to as  \cite[Theorem A]{TangL2}, which is \cite[Corollary to Theorem 3.1]{Bui}, actually uses the same $r_j$ as in Theorem~\ref{thm:riesz}, localized on the Fourier transform side.

In this section we consider these two types of localization for a more general convolution kernel, with minimal smoothness and decay conditions on the auxiliary functions used for the localization, and prove that the boundedness of the associated singular integral operators on $\hone(\Rn)$ does not depend on which type of localization we choose.  

\subsection{Calder\'on--Zygmund singular integral operators of convolution type}
We assume that our kernel is $\delta$-kernel, namely
 $K$ is a measurable function on $\Rn \setminus\lb 0\rb$ and there exist $\delta\in (0,1]$ and $C>0$ such that
\begin{equation}
\label{eqn:Kdecay}
 |K(x)|\leq \frac{C}{|x|^{n}}
\end{equation}
and
\begin{equation}
\label{eqn:Ksmoothness}
|K(x-y)-K(x)|\leq C\frac{|y|^{\delta}}{|x|^{n+\delta}} \quad \mbox{ for } 2|y|\leq |x|.
\end{equation}
We also assume that a cancellation condition holds in the following form:  
\begin{equation}
\label{eqn:Kcancellation}
\exists A  < \infty \quad \sup_{0<r<R}\bigg|\int_{r<|x|<R}K(x)dx\bigg| \leq A.
\end{equation}
We then define the associated convolution operator $T: \cS(\Rn) \ra \cS'(\Rn)$ by the principal value integral
$$Tf(x) = \lim_{\varepsilon \ra 0} \int_{|x- y| \geq \varepsilon} K(x - y) f(y) dy.$$
This extends to a bounded operator on $L^2(\Rn)$ (see \cite[Section VII.3]{SteinHA}).
Furthermore,  $T$ is bounded on $L^p(\Rn)$, $1<p<\infty$, and from $L^1(\Rn)$ to $L^{1,\infty}(\Rn)$. See \cite[Chapter 2]{stein_SIDPF}  for the classical theory of Calder\'on--Zygmund operators of convolution type and \cite[Chapter 4]{Grafakosbook2} for more general kernels.  A general Calder\'on--Zygmund operator can be written as a sum of an operator of this form plus an operator given by multiplication by a bounded function (see \cite[Section I.7.4]{SteinHA},\cite[Proposition 4.1.11]{Grafakosbook2}), but since multiplication by a bounded function is typically not a bounded operator on $\hone$, here we only consider the singular integral part.

We define two types of operators  via localizations of the kernel $K$.  On the one hand, we look at operators of the form
$$T^{\psi}(f):= T(f -\psi\ast f)$$
for suitable functions $\psi$.  These are modelled on Goldberg's localized Riesz transforms in Theorem~\ref{thm:riesz}, with $\psi = \check{\varphi}$.  

On the other, based on the definition of the local Hilbert transform \eqref{Hloc}, we look at operators $T_{\eta}$, associated to the kernel $K\eta$, for a class of functions $\eta$.  

\subsection{The Operator $T_{\eta}$}
We now show that under certain weak smoothness and decay conditions on $\eta$, the kernel $K\eta$ satisfies the same conditions as $K$ and so we can associate to it an operator $T_\eta$ in the same way that we associate $T$ to $K$, and this operator enjoys the same boundedness properties as $T$.
 
\begin{lemma} 
\label{Ketastd}
Suppose that $\eta$ is a bounded function and there exists a constant $C$ for which
\begin{equation} 
\label{Lipeta}
|\eta(x-y)-\eta(x)|\leq C\frac{|y|^{\delta}}{|x|^{\delta}} \quad \mbox{ for } 0 <2|y|\leq |x|.
\end{equation}
Then $K\eta$ satisfies conditions \eqref{eqn:Kdecay} and \eqref{eqn:Ksmoothness}.
\end{lemma}

\begin{proof}
We have that  $|K\eta|\leq \|\eta\|_\Linfty|x|^{-n}$,
and for $0 <2|y|\leq |x|$, since this implies $|x - y| \geq |x|/2$, we get
\begin{align*}
&|K(x-y)\eta(x-y)-K(x)\eta(x)|\\
&\leq |K(x-y)||\eta(x-y)-\eta(x)|+|K(x-y)-K(x)||\eta(x)|\\
&\lesssim \frac{1}{|x-y|^n} \frac{|y|^{\delta}}{|x|^{\delta}}+\frac{|y|^{\delta}}{|x|^{n+\delta}} \\
&\lesssim \frac{|y|^{\delta}}{|x|^{n+\delta}}.
\end{align*}
\end{proof}

The condition on $\eta$ is a local $\delta$-Lipschitz condition such that the Lipschitz constant in $B(x,|x|/2)$ decays like $|x|^{-\delta}$.   In particular, if $\eta$ has compact support, then we just require $\eta \in \Lip_\delta$.  

Next we show that $K\eta$ satisfies \eqref{eqn:Kcancellation}. For this we need further assumptions on $\eta$ guaranteeing that $\eta$ has a certain decay at infinity and $\eta - 1$ vanishes sufficiently fast at the origin.  Note that we do not require that $\eta \equiv 1$ in a neighborhood of the origin.  For example, if $\eta \in \Lip_\delta$ then it suffices that $\eta(0) = 1$.

\begin{lemma} 
\label{T1bddL2}
If $\eta$ satisfies,
\begin{equation} 
\label{decayeta}
\int_{0<|x|<1} \frac{|\eta(x)-1|}{|x|^n}dx +\int_{|x| \geq 1}\frac{|\eta(x)|}{|x|^n}dx < \infty,
\end{equation}
then \eqref{eqn:Kcancellation} holds for $K\eta$.
\end{lemma}
\begin{proof}
For $0 < r < R$, by the assumptions on $K$ and $\eta$,
\begin{align*}
&\bigg| \int_{r<|y|<R} K(y)\eta(y)dy\bigg|\\
&\leq \bigg| \int_{r<|y|<1} K(y)[\eta(y)-1]dy\bigg| + A +\bigg| \int_{1\leq |y|<R} K(y)\eta(y)dy\bigg| \\
&\lesssim  \int_{0<|y|<1} \frac{|\eta(y)-1|}{|y|^n}dy+  \int_{|y| \geq 1} \frac{|\eta(y)|}{|y|^n}dy + 1
\leq A',
\end{align*}
where the constant $A'$ is independent of $r$ and $R$.
\end{proof}

If $\eta$ satisfies the hypotheses of both lemmas, we get the boundedness of $T_{\eta}$ on $L^2(\Rn)$.  Once we have $L^2$-boundedness,  we can show that $T_{\eta}$ is bounded on $L^p(\Rn)$ for $1<p<\infty$ and from $L^1(\Rn)$ to weak-$L^1(\Rn)$.

\subsection{Equivalent boundedness of $T_{\eta}$ and $T^{\psi}$}
For nice choices of $\psi$, one can study the boundedness of the localized operator $T^{\psi}$ as a pseudo-differential operator.  Alternatively, one can show that the kernel of $T^{\psi}$ is sufficiently good.  In particular, assuming that $\psi\in L^1$,
 $T^{\psi}$ inherits from $T$ the boundedness on $L^p$ for $1<p<\infty$ and the weak $(1,1)$ boundedness. We are interested in the boundedness on $\hone$.

Instead of working with $T^\psi$ directly, we give sufficient conditions on $\psi$ to be able to obtain the boundedness of $T^\psi$ on $\hone$ from the boundedness of $T_\eta$.  As a consequence of the following equivalence, it suffices to prove the boundedness for one good $\eta$ (e.g. a $\Lip_\delta$ function of compact support) in order to obtain it for the whole class of $\eta$ and $\psi$.

\begin{theorem}
\label{Mainthm1}
Suppose $\eta$ satisfies the hypotheses of Lemmas~\ref{Ketastd} and \ref{T1bddL2},  and $\psi$ satisfies
$$\psi\in L^1(\Rn) \cap L^2(\Rn), \quad 
\int \psi = 1,$$
\begin{align}
\label{eqn:decaypsi}
\int_{|x|\geq 1}\lsk \frac{1}{|x|^n}\int_{|y|\geq |x|/2}|\psi(y)|dy + \frac{1}{|x|^{n+\delta}}\int_{|y|\leq |x|/2} |y|^{\delta}|\psi(y)| dy \rsk dx < \infty
\end{align}
and
\begin{align}
\label{eqn:smoothpsi}
\int_{|x|\geq 1}\lsk \int_{|y|\leq |x|/2}\frac{|\psi(x-y)-\psi(x)|}{|y|^n} dy\rsk dx < \infty.
\end{align}
Then the operator $T_{\eta}$ is bounded on $h^1(\Rn)$ if and only if $T^{\psi}$ is bounded on $h^1(\Rn)$.
\end{theorem}

Note that condition \eqref{eqn:decaypsi} will hold, for example, if  the function $|x|^\eta \psi(x)$ is integrable for some $\eta > 0$, while \eqref{eqn:smoothpsi} will hold if  $\psi \in \Lip_\alpha$ for $0 < \alpha \leq 1$ and $|x|^\alpha \|\psi\|_{\Lip_\alpha(B(x,|x|/2))}$ is integrable.

\begin{proof}
We will show that the operator  $T_E:=T_{\eta}-T^{\psi}$ is bounded on $L^1(\Rn)$, and from that deduce that $T_E$ is bounded on $h^1(\Rn)$.

For $0 < \varepsilon < 1/2$, let $K_\varepsilon(x) = \chi_{|x| > \varepsilon} K(x)$ be the usual truncation of the kernel.  Applying the same to $K\eta$, we consider  
$(K\eta)_\varepsilon -K_\varepsilon+K_\varepsilon\ast\psi$.
Set
$$K_*(x) = \sup_{\varepsilon > 0} |(K\eta)_\varepsilon(x) -K_\varepsilon(x)+K_\varepsilon\ast\psi(x)|.$$
We claim that $K_* \in L^1(\Rn)$.

For the local estimate, we write 
$$K_*(x) \leq   |K(x)(\eta(x) - 1)| +  T_*(\psi)(x),$$
 where $T_*(\psi)(x) = \sup_{\varepsilon > 0}|K_\varepsilon\ast\psi(x)|$.  Recall that the maximal operator $T_*$ has the same boundedness as $T$ (see \cite[Section I.7]{SteinHA},\cite[Section 4.2.2]{Grafakosbook2}).  Thus we have 
$$
\int_{|x| < 1} K_*(x) dx \lesssim \int_{|x| < 1} \frac{|\eta(x) - 1|}{|x|^n}dx +  \|\psi\|_{L^2} 
< \infty
$$

For $|x| \geq 1$, we can write
$$K_*(x) \leq  |K\eta(x)| + \sup_{\varepsilon > 0} |K(x)- K_\varepsilon\ast\psi(x)|.$$
The first term is integrable for $|x| \geq 1$ by condition \eqref{decayeta} on $\eta$, so it remains to bound the integral of the second term.  We do this by fixing $\varepsilon$ and obtaining a pointwise estimate on $|K(x)- K_\varepsilon\ast\psi(x)|$ in terms of $\psi$.

Using the hypotheses on $\psi$, we can write, for $|x| \geq 1$ and $\varepsilon < 1/2$,
\begin{align}
\nonumber
\bigg|K(x) - \int_{|x-y|>\varepsilon} K(x)\psi(y)dy\bigg| 
\nonumber
&\leq |K(x)|\int_{|x-y|\leq\varepsilon} |\psi(y)|dy\\
\label{eqn:Kpsi}
&\lesssim |x|^{-n}\int_{|y| \geq |x|/2} |\psi(y)|dy.
\end{align}
Thus it remains to consider
\begin{align*}
&\int_{|x-y|>\varepsilon} [K(x)-K(x-y)]\psi(y)dy \\
&= \int_{\bigg\lb \substack{|y|\leq |x|/2\\ |x-y|\geq |x|/2}\bigg\rb} 
+\int_{\bigg\lb \substack{|y|\geq |x|/2\\ |x-y|\geq |x|/2}\bigg\rb} 
+\int_{\varepsilon<|x-y|\leq |x|/2} [K(x)-K(x-y)]\psi(y)dy\\
& =:I_1+I_2+I_3.
\end{align*}

To estimate $I_1$, note that $B(0,|x|/2) \subset B(x,|x|/2)^c$ so we can use the smoothness condition \eqref{eqn:Ksmoothness} on $K$ to write 
\begin{align*}
|I_1|&\leq \int_{|y|\leq |x|/2} |K(x)-K(x-y)||\psi(y)|dy \\
&\lesssim  \frac{1}{|x|^{n+\delta}}\int_{|y|\leq |x|/2} |y|^{\delta}|\psi(y)| dy.
\end{align*}
For $I_2$, we use the decay bound \eqref{eqn:Kdecay} on $K$  to get, as in \eqref{eqn:Kpsi},
\begin{align*}
|I_2|&\leq\int_{\bigg\lb \substack{|y|\geq |x|/2\\ |x-y|\geq |x|/2}\bigg\rb} (|K(x)||\psi(y)|+|K(x-y)||\psi(y)|)dy \\
&\lesssim \frac{1}{|x|^{n}} \int_{|y|\geq |x|/2}|\psi(y)|dy.
\end{align*}
Finally, for $I_3$, we again use the decay bound \eqref{eqn:Kdecay} as well as the cancellation condition \eqref{eqn:Kcancellation}  on $K$ to obtain
\begin{align*}
|I_3|&\leq \bigg|\int_{\varepsilon<|x-y|\leq |x|/2} \lsk K(x)\psi(y)-K(x-y)(\psi(y)-\psi(x))\rsk dy\bigg| \\
&\quad\quad\quad\quad + \bigg|\int_{\varepsilon<|x-y|\leq |x|/2} K(x-y)\psi(x)dy\bigg|\\
&\lesssim \int_{|y|\geq |x|/2} \frac{|\psi(y)|}{|x|^n}dy + \int_{|x-y|\leq |x|/2}   \frac{|\psi(y)-\psi(x)|}{|x-y|^n} dy  + |\psi(x)|\\
&\lesssim \frac{1}{|x|^n}\int_{|y|\geq |x|/2}|\psi(y)|dy +\int_{|y'|\leq |x|/2}\frac{|\psi(x-y')-\psi(x)|}{|y'|^n} dy +   |\psi(x)|.
\end{align*}

Since all these estimates are independent of $\varepsilon$, we can combine \eqref{eqn:Kpsi}, the bounds on $I_1$, $I_2$ and $I_3$, and the integrability conditions \eqref{eqn:decaypsi} and \eqref{eqn:smoothpsi} on $\psi$ to get that
$$\int_{|x| \geq 1}\sup_{\varepsilon > 0} |K(x)- K_\varepsilon\ast\psi(x)|< \infty$$
and as a result, $K_* \in L^1(\Rn)$.
Let 
$$K_E(x) = \lim_{\varepsilon \ra 0} ((K\eta)_\varepsilon -K_\varepsilon+K_\varepsilon\ast\psi)(x) = K\eta(x) -K(x)+ \lim_{\varepsilon \ra 0}K_\varepsilon\ast\psi(x).$$  
This limit exists for almost every $x \in \Rn$ by the properties of $T_*$ (see \cite[p.\ 45]{stein_SIDPF}).  By the Dominated Convergence Theorem, $K_E \in L^1(\Rn)$ and
$$\lim_{\varepsilon \ra 0} ((K\eta)_\varepsilon -K_\varepsilon+K_\varepsilon\ast\psi) \ast f = K_E \ast f \quad \mbox { for } f \in L^1(\Rn).$$

Let $f\in \hone(\Rn)$.  Then
$T_E f \in L^1(\Rn)$ with 
$$\|T_E f\|_{L^1(\Rn)} \leq \|K_E\|_{L^1(\Rn)} \|f\|_{L^1(\Rn)}.$$
Moreover,  the localized Riesz transforms $r_j$, $j = 1,\ldots, n$ of Goldberg (see Theorem~\ref{thm:riesz}) commute with convolution with the $L^1$ function $K_E$, so
$$r_j(T_E(f))=T_E(r_j(f)).$$
Since $r_j(f)\in L^1(\Rn)$,  hence $T_E(r_j(f))\in L^1(\Rn)$, we have
\begin{align*}
\|T_E f\|_{\hone(\Rn)} &\lesssim \|T_E f\|_{L^1(\Rn)} + \sum_{j = 1}^n \|r_j(T_E(f))\|_{L^1(\Rn)}\\
&  \leq \|K_E\|_{L^1(\Rn)}(\|f\|_{L^1(\Rn)} + \sum_{j = 1}^n \|r_j(f)\|_{L^1(\Rn)} )\\
&\lesssim \|f\|_{\hone(\Rn)}.
\end{align*}
\end{proof}

Starting with the kernels of the Riesz transforms $R_j$, $j = 1,\ldots, n$, which satisfy conditions \eqref{eqn:Kdecay} and \eqref{eqn:Ksmoothness} with $\delta = 1$ and \eqref{eqn:Kcancellation} with $A = 0$, and noting that the function $\psi = \check{\varphi}\in \cS(\Rn)$ for $\varphi$ as in Theorem~\ref{thm:riesz} satisfies the hypotheses of Theorem~\ref{Mainthm1}, we can use both theorems to get the following.
\begin{corollary} 
\label{localriesz}
Suppose that $\eta$ and $\psi$ satisfy assumptions in Theorem~\ref{Mainthm1}. For $j = 1,\ldots, n$, define, for $f \in \cS(\Rn)$,
\begin{align*}
\cRjeta f(x)&:= c_n \lim_{\varepsilon \ra 0} \int_{|x - y| \geq \varepsilon} \frac{(x_j-y_j)\eta(x-y)}{|x-y|^{n+1}}f(y)dy \\
\text{and}\\
\cRjpsi f(x)&:= R_j(f - \psi \ast f).
\end{align*}
Then
\begin{enumerate}
\item $\cRjeta$ and $\cRjpsi$ both map $h^1(\Rn)$ to $L^1(\Rn)$;
\item for $f\in L^1(\Rn)$, $f\in h^1(\Rn) \iff \cRjeta(f)\in L^1(\Rn), j=1,...,n \iff \cRjpsi(f)\in L^1(\Rn), j=1,...,n.$
\end{enumerate}
\end{corollary}

As seen above, because convolution operators commute with the $r_j$, the boundedness from $\hone(\Rn)$ to $L^1(\Rn)$ implies the boundedness on $\hone(\Rn)$, and extra cancellation conditions on the kernel are not needed (see also the remark following the proof of \cite[Theorem 4]{Goldberg2}).  

\subsection{Inhomogeneous singular integral operators}
We end this section by introducing a class of operators which are not necessarily of convolution type.  As can be seen from the definition, however, for a convolution $\delta$-kernel $K$ and an $\eta$ which decays sufficiently fast at infinity, the localized kernel $K\eta$ will be in this class.

Let $\Delta:=\lb (x,x): x\in \Rn\rb$ denote the diagnoal in $\R^{2n}$.

\begin{definition}
\label{def:inhom}
Let $\delta \in (0,1]$.  We say $T$ is an he inhomogeneous singular integral operator associated with  an inhomogeneous $\delta$-kernel $K$ if $K$ is a locally integrable function on $\R^{2n}\setminus\Delta$ and there exist $\varepsilon>0$ and $C$ such that 
\begin{enumerate}
\item $|K(x,y)|\leq C\min\lb |x-y|^{-n}, |x-y|^{-n-\varepsilon}\rb$ for all $(x,y)\in \R^{2n}\setminus \Delta$;
\item $|K(x,y)-K(x,z)|+|K(y,x)-K(z,x)|\leq C |y-z|^{\delta}|x-y|^{-n-\delta}$ for $2|y-z|\leq |x-y|$.
\end{enumerate}
In addition, we assume that the operator $T$, defined by 
$Tf(x):=\int_{\Rn}K(x,y)f(y)dy$
for $f \in \Cinftyc(\Rn)$ and $x \not\in \supp(f)$, can be extended to a bounded operator on $L^2(\Rn)$.
\end{definition}

This type of operator was introduced in \cite{DHZ} and a generalization  was considered in \cite{DLPV1}, where it was shown that if certain cancellation conditions are imposed on $T$,  then it is bounded from $h^p(\Rn)$ to $h^p(\Rn)$, $0 < p \leq 1$.  If no cancellation conditions are imposed, we still have boundedness from $\hone(\Rn)$ to $L^1(\Rn)$, as can be seen by noting that $\hone$ atoms are mapped to ``pre-molecules" which are bounded in $L^1(\Rn)$ (see \cite[Definition 3]{DLPV2}).
 As $T$ is a singular integral operator, $T$ is also bounded from $L^p(\Rn)$ to itself for $1<p<\infty$.  We will consider the commutators of these operators with $\bmo$ functions in Section~\ref{sec:commutators}.
 
\section{Local Hardy spaces related to commutators with bmo functions}
\label{sec:spaces}
In this section we introduce some ``local" analogues of commutator maximal function defined in \cite{Ky}, and the atoms in \cite{Perez}, and consider the relationships between them.

\subsection{Some Maximal Functions}
\label{sec:maximal}
Recall that the Hardy-Littlewood Maximal function defined for $f \in \Loneloc(\Rn)$ by
$$Mf(x):=\sup_{B\ni x}|f|_{B}$$
 is bounded on $L^p(\Rn)$ for $1<p\leq\infty$ and also bounded from $L^1(\Rn)$ to $L^{1,\infty}(\Rn)$. 

In connection with commutators, the following two maximal functions were introduced in \cite[Section 6]{CRW} in the one-dimensional case.
Setting
$$M(b,f)(x):=\sup_{I\ni x} \big|(b(x) - b_I)f_{I}\big|,$$
it was shown in \cite[Theorem IX]{CRW} that $b \in \BMO(\R)$ if and only if $f \mapsto M(b,f)$ is bounded on $L^2(\R)$.
The proof uses  \cite[Theorem X]{CRW} which states that for a different maximal function,
$$N(b,f)(x):=\sup_{I\ni x} |(bf)_I - b(x)f_{I}|,$$
$f \mapsto M(b,f)$ is bounded on $L^2(\R)$ when $b \in \BMO(\R)$.

The commutator of $M$ with multiplication by $b$, that is
$$[M,b](f)(x) := M(bf)(x) - b(x)Mf(x),$$
was studied in \cite{MilmanSchonbek, BasteroMilmanRuiz, Amiran} and shown to be bounded on $L^p(\Rn)$ for $1<p<\infty$ if and only if $b\in BMO(\Rn)$ and $b$ is bounded below.
 
Finally, one has the following maximal function,
$$M_b f(x)=\sup_{B\ni x} \frac{1}{|B|}\int_{B}|b(x)-b(y)||f(y)|dy,$$
which was studied in \cite{SegoviaTorrea, GHST, Amiran} and is sometimes denoted by $C_b$.
In this case, as in \cite{CRW}, one has that $M_b$ is bounded on $L^p(\Rn)$ for $1<p<\infty$ if and only if $b\in BMO(\Rn)$. Moreover, for $b\in BMO(\Rn)$, \cite[Corollary 1.11]{Amiran} gives the pointwise domination $M_b f (x) \lesssim \|b\|_\BMO M^2 f(x)$ for $f \in \Loneloc(\Rn)$, and therefore
\begin{equation}
\label{eqn:MbBound}
\|M_b\|_{L^p \mapsto L^p} \lesssim \|b\|_\BMO.
\end{equation}

Ky \cite{Ky} introduced the following commutator maximal function definition of the space  $H^1_b$.
\begin{definition}[\cite{Ky}, Definition 2.2]
\label{def:Ky}
Let $b \in \BMO(\Rn)$ be nontrivial.  The space $\Honeb(\Rn)$ consists of $f \in \Hone(\Rn)$ such 
$[b,\fMnt] f \in \Lone(\Rn)$, where
$[b,\fMnt]f(x): = \fMnt[b(x) f(\cdot) - b(\cdot) f(\cdot)](x)$
and $\fMnt$ is the non-tangential grand maximal functions defined by
$$\fMnt f(x):=\sup\{ |f * \phi_t(y)|: |y - x| < t, \varphi\in \cA\}$$
with
$$\cA = \{\phi\in \cS(\Rn): \|\phi\|_{\infty}  + \|\grad \phi\|_{\infty} \leq (1 + |x|^2)^{-n-1}\}.$$
The norm on $\Honeb$ is given by 
$$\|f\|_{\Honeb}:= \|f\|_\Hone \|b\|_\BMO + \|[b,\fMnt]f\|_{L^1}.$$
\end{definition}
Ky also defines a ``local" version of the non-tangential grand maximal function $\fMnt$ by restricting $t$ to $(0,1)$.  As pointed out following Proposition~\ref{prop:h^1}, this is equivalent to the maximal functions $\cM$ and $\fM$ in terms of characterizing $\hone(\Rn)$.  However, there is no definition of $\honeb$ given in \cite{Ky}. 

We choose the grand maximal function $\cM$, using test functions of compact support, to give the analogue of Ky's definition in the local case.
\begin{definition} 
\label{def:cmbf}
Given $b\in L^2_{loc}(\Rn)$, we define the maximal commutator function for $f\in L^2_{loc}(\Rn)$ to be 
$$
\cMb f(x):= [b,\cM](x): = \cM[b(x) f(\cdot) - b(\cdot) f(\cdot)](x).
$$
\end{definition}
Recall that in Proposition~\ref{prop:h^1} (for $T = 1$) we defined, for $f \in \cS'(\Rn)$,
$$\cM(f)(x):=\sup_{\phi} |\langle f, \phi \rangle|,$$
where the supremum is taken over all $\phi\in C^1(\Rn)$ with $\supp(\phi)\subset B(x,t)$ for some $0<t<1$, $\|\phi\|_{L^{\infty}}\leq |B(x,t)|^{-1}$ and $\|\nabla\phi\|_{L^{\infty}}\leq [t|B(x,t)|]^{-1}$.
We therefore have that for $b, f \in L^2_{loc}(\Rn)$,
$$\cMb f(x)= \sup_{\phi} \big|\langle b(x) f(\cdot) - b(\cdot) f(\cdot), \phi \rangle\big| = \sup_{\phi} \bigg|\int_{\Rn} [b(x)  - b(y)] f(y)\phi(y) dy\bigg|.$$
Since $(b(x) - b)f \in L^1_{loc}(\Rn)$ for almost every $x \in \Rn$,  the integral is well-defined almost everywhere, and by the size and support conditions on $\phi$,
\begin{equation}
\label{eqn:cMb}
\cMb f(x) \lesssim M_b f(x).
\end{equation}
The extension of this definition from $f \in L^2_{loc}(\Rn)$ to general $f \in \hone(\Rn)$, except in the special case where $b$ is bounded, is not obvious.  In fact, this is also the case for Ky's definition, Definition~\ref{def:Ky}, and additional work, such as the sublinear decomposition \cite[Theorem 3.1]{Ky}, is needed there.

\subsection{Atomic commutator Hardy spaces}
In \cite{Perez}, P\'erez defined a restricted notion of atom with cancellation against a function $b\in \BMO$.  
\begin{definition}[\cite{Perez}, Definition 1.3]
\label{def:Perez}
For  $b \in \BMO(\Rn)$, a $b$-atom $a$ such that for some cube $Q$ 
\begin{itemize}
\item[(i)] $\supp(a) \subset Q$; 
\item[(ii)] $\|a\|_\Linfty \leq |Q|^{-1}$;
\item[(iii)] $\int a= 0$ and $\int ab  = 0$.
\end{itemize}
\end{definition}
P\'erez then proceeded to define an atomic space, which he called $\Honeb$ (Ky  \cite{Ky} uses $\cHoneb$ for this space) by taking all $f \in \Lone$ which can be written as $f=\sum_{j=1}^{\infty}\lambda_{j}a_j$, where $a_j$ are $b$-atoms and $\lb \lambda_j\rb\in \ell^1$. Note that since $b$-atoms are also $\Hone$ atoms, such a decomposition will converge in the  $\Hone$ norm with $\|f\|_\Hone \leq \sum |\lambda_j|$, and hence P\'erez's space is contained in $\Hone$.

We give the following straightforward adaptation of P\'erez's definition to the local case.
\begin{definition}
\label{def:localPerez}
Fix $b \in \bmo(\Rn)$. We say $a$ is a P\'erez $\honeb$ atom if for some ball $B(x_0, r)$ we have
\begin{enumerate}
    \item $\supp(a)\subset B(x_0,r)$
    \item $\|a\|_{L^{2}(\Rn)}\leq |B(x_0,r)|^{-\frac12}$
    \item  $\ds \int a=0$ and  $\ds \int ab\; =0$ whenever $r<1$.
\end{enumerate}
For $f \in \Lone(\Rn)$, we say $f\in \honeperb(\Rn)$ if $f$ can be written as $f=\sum_{j=1}^{\infty}\lambda_{j}a_j$, where $a_j$ are P\'erez $\honeb$ atoms and $\lb \lambda_j\rb\in \ell^1$. For such functions, we define
$$\|f\|_{\honeperb}:=
\inf\sum_{j}|\lambda_j|,$$
where the infimum is taken over all such decompositions.
\end{definition}
As in the homogeneous case, we note that a P\'erez $\honeb$ atom is a $(1,2)$ atom in $\hone$ and therefore the decomposition converges in $\hone$ and we have $\|f\|_\hone \leq \sum |\lambda_j|$.  Moreover, each P\'erez $\honeb$ atom supported in a ball $B = B(x_0,r)$ also satisfies that $a(b-c_{B})\in \hone$, where $c_B$ is as in \eqref{eqn: cB}.  Indeed, we have that
$a(b-c_{B})$ is also supported in $B$, has integral zero whenever $r < 1$, and by \eqref{eqn:bmop}
\begin{align}
\nonumber
\|a(b-c_{B})\|_{L^{s}(\Rn)} & \leq \|a\|_{L^{2}(\Rn)} \|b-c_{B}\|_{L^{p}(\Rn)}\\
\label{eqn:abc}
& \leq |B(x_0,r)|^{-\frac12}\|b\|_\bmop |B(x_0,r)|^{\frac1p} = \|b\|_\bmop |B(x_0,r)|^{\frac1s - 1}
\end{align}
for $1\leq s < 2$ and $\frac 12 + \frac 1p= \frac1s$; taking $s> 1$ makes $a(b-c_{B})$ a multiple of a $(1,s)$ atom and we have
$$\|a(b-c_{B})\|_\hone \lesssim \|b\|_\bmop.$$

We now give a variation on this definition which involves approximate cancellation conditions.

\begin{definition} 
\label{def:Approxh1bAtom}
Fix $b \in \bmo(\Rn)$ and let $C_b = \|b\|_{\bmo,2}$ as defined in \eqref{eqn:bmop}.  We say $a$ is an approximate $\honeb$ atom if for some ball $B(x_0, r)$ we have
\begin{enumerate}
    \item $\supp(a)\subset B(x_0,r)$;
    \item $\|a\|_{L^2(\Rn)}\leq |B(x_0,r)|^{-\frac12}$;
  \item $\ds \bigg|\int a\bigg|\leq \frac{1}{[\log(1+r^{-1})]^{2}}$ and  $\ds \bigg|\int a(b-c_B)\bigg|\leq \frac{C_b}{\log(1+r^{-1})}$.
\end{enumerate}
For $f \in \Lone(\Rn)$, we say $f\in \honeatomb(\Rn)$ if $f$ can be written as $f=\sum_{j=1}^{\infty}\lambda_{j}a_j$, where $a_j$ are approximate  $\honeb$ atoms and $\lb \lambda_j\rb\in \ell^1$. For such functions, we define
$$\|f\|_{\honeatomb}:= 
\inf\sum_{j}|\lambda_j|,$$
where the infimum is taken over all such decompositions.
\end{definition}
The choice of  $C_b = \|b\|_{\bmo,2}$ guarantees that, by Cauchy-Schwarz, condition (2) implies condition (3) when $r \geq 1$.  Thus every P\'erez $\honeb$ atom  is an approximate $\honeb$ atom, and
$$\honeperb(\Rn) \subset \honeatomb(\Rn).$$
 Moreover, comparing with 
 Definition~\ref{def:R-atom} for $R = 1$ and $q = 2$, we see, as in \eqref{eqn:abc}, that not only is every approximate $\honeb$ atom $a$ an approximate $(1,2)$ atom (up to a factor of $\log 2$), but also $a(b-c_{B})$ is a multiple of an approximate $(1,s)$ atom for $1< s < 2$ and $\frac 12 + \frac 1p= \frac 1s$, with
\begin{equation}
\label{eqn:honeb-abc}
\|a(b-c_{B})\|_\hone \lesssim \|b\|_\bmop.
\end{equation}

Recall that in $\hone$ (or $h^p$), every atom (or molecule) with approximate cancellation conditions can be written as an infinite linear combination of atoms with exact cancellation - see \cite{DY, DLPV1}. Is it possible to decompose  an approximate $\honeb$ atom into P\'erez $\honeb$ atoms with exact cancellation? 

In the following proposition, we give a partial converse showing that every approximate $\honeb$ atom can be written as a finite sum of atoms of integral zero, but we do not achieve the exact cancellation against $b$, and it is not clear if this can be done.

\begin{proposition}
\label{prop:decomoph1b}
Suppose $A$ is an approximate $\honeb$ atom with respect to the ball $B(x_0, r)$, and $r < 1$.  Then  $\ds A = \sum \lambda_j a_j$ where the sum is {\bf finite}, $a_j$ are approximate $\honeb$ atoms supported in balls $B_j$, respectively, and in addition, when $r(B_j) < 1$,
$$\int a_j = 0.$$
Moreover, for some constant depending only on $n$,
$\ds \sum |\lambda_j| \leq C_{n}$.
\end{proposition}

\begin{proof}
Set $\alpha = \int A \neq 0$ (otherwise we are done).  Denote $B(0, 2^jr)$ by $B_j$, $j = 0, 1, 2,\ldots, k$, where $2^{k-1}r<1\leq 2^k r$, and let $\eta_j = \frac{\chi_{B_j}}{|B_j|}$ be the normalized characteristic function of $B_j$.  Then
$$
A=A - \alpha\eta_0+\sum_{j=1}^k \alpha (\eta_{j-1} - \eta_j) + \alpha\eta_k = \sum_{i = 0}^{k+1} A_j
$$
where 
\begin{align*}
& A_0:=A - \alpha\eta_0\\
&A_j :=\alpha(\eta_{j-1} - \eta_j), \quad j = 1, \ldots, k, \;\mbox{ and}\\
& A_{k+1}:= \alpha\eta_k.
\end{align*}

First observe that $\supp(A_j)\subset B_j$ for $j=1,...,k$, and $\supp(A_{k+1})\subset B_k$.  Moreover, since $\int \eta_j = 1$ for all $j$, we have that $\int A_j = 0$ for $j=0,...,k$.  

Write $A_j = \lambda_j a_j$ where 
$$a_0:= \frac{1}{2}A_0, \quad \lambda_0 := 2, \quad  a_{k+1}:= \alpha^{-1}A_{k+1}, \quad  \lambda_{k+1}:= \alpha,$$
and
$$a_j :=\frac{A_j}{\alpha 2^{n}\log(1+[2^jr]^{-1})}, \quad \lambda_j := \alpha 2^{n}\log(1+[2^jr]^{-1}), \quad j = 1, \ldots, k.$$
Since $|\alpha| \leq \|A\|_{\Lone} \leq 1$, $\|\eta_j\|_{L^\infty} \leq |B_j|^{-1}$, and $2^{n}\log(1+[2^jr]^{-1}) > 1$ for $j = 1, \ldots, k$,
we have that 
$\|a_0\|_{L^2}\leq |B_0|^{-\frac12}$, $\|a_j\|_{L^{\infty}}\leq |B_j|^{-1}$ for $j=1,...,k$, and $\|a_{k+1}\|_{L^\infty}\leq |B_k|^{-1}$. Thus, all the $a_j$ satisfy Conditions (1) and (2) in Definition~\ref{def:Approxh1bAtom} and in addition $\int a_j = 0$ for $j=0,...,k$.  Since $r(B_{k}) = 2^{k}r \geq 1$, the latter condition is not required for $a_{k+1}$.

To verify Condition (3) in Definition~\ref{def:Approxh1bAtom} , we need to check that the $a_j$ satisfy the approximate cancellation condition against $b$. For $j = 0$, we get the approximate cancellation for $a_0$ from that of $A$:
$$
\bigg|\int a_0b\bigg| =  \frac 12\bigg|\int Ab- \alpha \int b \eta_0\bigg| = \frac 12\bigg|\int A[b-b_B]\bigg|\leq \frac{C_b}{\log(1+r^{-1})}.
$$
For $j=1,..., k$, we have, using \eqref{eqn:doubling},
\begin{align*}
\bigg|\int a_j b \bigg|&=  \frac{1}{2^{n}\log(1+[2^jr]^{-1})}\bigg| \int b(\eta_{j-1} - \eta_j)  \bigg|\\
&=  \frac{1}{2^{n}\log(1+[2^jr]^{-1})} |b_{2^{j-1}B}-b_{2^{j}B}| \\
& \leq  \frac{\|b\|_{\bmo}}{\log(1+[2^jr]^{-1})},
\end{align*}
and for $j = k+1$, since $r(B_k) = 2^kr \geq 1$,
$$
\bigg|\int a_{k+1} b\bigg| = |b_{B_k}| \leq  \frac{\|b\|_{\bmo}}{\log(1+(2^kr)^{-1})}.
$$
If, as noted after Definition~\ref{def:Approxh1bAtom}, we take $C_b = \|b\|_{\bmo, 2} \geq \|b\|_{\bmo}$,
then we can conclude that $a_j$ for $j=0,...,k+1$ are all approximate $\honeb$ atoms.

Finally, we have
\begin{align*}
\sum |\lambda_j|&= 2+\alpha 2^{n} \sum_{j=1}^k \log(1+[2^jr]^{-1}+ \alpha\\
&\leq 3+\frac{2^{n}}{[\log(1+r^{-1})]^2}  \sum_{j=1}^k \log(1+[2^jr]^{-1})\\
&\leq 3+\frac{2^{n}  k \log(1+r^{-1})}{[\log(1+r^{-1})]^2}\\
&\leq 3+ \frac{2^{n}(\log_2 r^{-1}+1)}{[\log(1+r^{-1})]} \\
&\leq  C_{n}.
\end{align*}
\end{proof}

\subsection{Relations between the spaces}
For the homogeneous case, Ky shows (see \cite[Theorem 5.2]{Ky}) that
$$\cHoneb(\Rn) \subset \Honeb(\Rn) \subset \Hone(\Rn),$$
where $\cHoneb$ is P\'erez atomic space and $\Honeb$ is the maximal commutator space, but does not discuss whether the inclusion is proper.  The result of \cite{HST} shows that the last inclusion is proper, at least in dimension $1$, unless $b$ is a constant.  

For the nonhomogeneous case, the discussion following the definitions of the atomic spaces in the previous section gave us the inclusions
$$\honeperb(\Rn)\subset\honeatomb(\Rn)\subset\hone(\Rn).$$
In the trivial case when $b$ is a constant function, the cancellation conditions against $b$ reduce to the conditions on the integral of the atoms.  Thus every $\hone$ atom with exact cancellation is a P\'erez $\honeb$ atom, so
\begin{equation}
\label{eqn:bconstant}
\honeperb(\Rn)=\honeatomb(\Rn)=\hone(\Rn).
\end{equation}
In the next section will see later that the second equality can hold for $b$ in a nontrivial subspace of $\bmo$.

While we have not defined a maximal space, we can consider the action of the commutator maximal function $\cMb$ on atoms.
\begin{proposition}
\label{h1binc}
Let $b\in \bmo(\Rn)$. Then for every $\honeb$ atom $a$,
$$\|\cM_b a\|_{L^1} \lesssim  \|b\|_\bmo.$$
\end{proposition}

\begin{proof}
Let $a$ be an approximate $\honeb$ atom with support in $B=B(x_0,r)$. We want to bound $\cM_b(a)$. 

For the local estimate, using \eqref{eqn:cMb}, we can apply the  $L^p$ boundedness of the maximal commutator operator 
$M_bf$
defined in Section~\ref{sec:maximal}, with the bound \eqref{eqn:MbBound}. Thus, using the $L^2$ size condition on  $a$ we have
$$
\int_{2B} \cMb(a) 
\leq |2B|^{1/2} \|M_b a\|_{L^2(2B)} \lesssim  |B|^{\frac12} \|a\|_{L^2} \|M_b\|_{L^2 \mapsto L^2} \lesssim \|b\|_\bmo.
$$

Next we handle the integral on $(2B)^c$. Note that
$$\bigg|\int_{B(x,t)} \phi(y)[b(x)-b(y)]a(y)dy\bigg| \neq 0$$
implies that there exists $y\in B(x,t)\cap B(x_0,r)$, which in turn implies
\begin{equation}
\label{eqn:intersection}
r\leq \frac{|x-x_0|}{2} \leq  |x-x_0|-|x_0-y|\leq |x-y|\leq t< 1. 
\end{equation}
This cannot happen when $r\geq 1$, so in that case the integral vanishes for all test functions $\phi$ and $\cM_b(a) = 0$ on $(2B)^c$.

For $r<1$, we may assume, by Proposition~\ref{prop:decomoph1b}, that $\int a = 0$.  Thus we have, for almost all $x \in (2B)^c$,
\begin{align}
\nonumber
&\bigg|\int [b(x)-b(y)]a(y) \phi(y)dy\bigg| \\
\nonumber
&\leq \bigg|\int [b(x)-b(y)]a(y)[\phi(y)-\phi(x_0)]dy\bigg|+ \bigg|\phi(x_0)\int b(y)a(y)dy\bigg|\\
\nonumber
&\leq  |b(x)-b_B|\|\grad \phi\|_\Linfty r \int|a(y)|dy+ \|\grad \phi\|_\Linfty r\int |b(y)-b_B| |a(y)|dy\\
\nonumber
&\quad\quad \quad + \| \phi\|_\Linfty \bigg|\int b(y)a(y)dy\bigg|\\
\label{eqn:not2B}
&\lesssim \frac{|b(x)-b_{B}|r}{|x-x_0|^{n+1}} + \frac{r\|b\|_{\bmo,2}}{|x-x_0|^{n+1}}+ \frac{C_b}{|x-x_0|^{n}\log(1+r^{-1})},
\end{align}
where in the last step we used the conditions on $\phi$, \eqref{eqn:intersection}, the $L^1$ estimate on $a(b - b_B)$ (see \eqref{eqn:abc} with $s = 1$, $p = 2$), and condition (3) in Definition~\ref{def:Approxh1bAtom}.
Since the estimate above is independent of $\phi$, it holds for $\cMb(a)(x)$ and therefore, again using \eqref{eqn:intersection},
\begin{align}
\nonumber
\int_{(2B)^c} \cMb(a)(x)  dx 
&=  \int_{2r<|x-x_0|<2}  \cMb(a)(x) dx \\
\label{eqn:not2Bint}
&\lesssim \int_{(2B)^c}\frac{|b(x)-b_{B}|r}{|x-x_0|^{n+1}} + \int_{(2B)^c}\frac{r\|b\|_{\bmo,2}}{|x-x_0|^{n+1}}\\
\nonumber
&\quad\quad\quad\quad+ \int_{2r<|x-x_0|<2}\frac{C_b}{|x-x_0|^{n}\log(1+r^{-1})}\\
\nonumber
&\lesssim \|b\|_{\bmo}.
\end{align}
Here we have used \eqref{eqn:bintegral}, \eqref{eqn:bmop}, and the fact that $C_b = \|b\|_{\bmo,2}$.
\end{proof}

If we take an $\hone$ atom $a$ without assuming any cancellation against $b$, the boundedness of the maximal function $\cMb$ turns out to be equivalent to whether $a(b-c_{B})$ belongs to $\hone$, where $B$ is the ball containing the support of $a$.  Such boundedness gives us automatically the approximate cancellation against $b$.

\begin{proposition} 
\label{ainh1bandh1}
Let $b \in \bmo(\Rn)$, $b$ nontrivial. Suppose $a$ is an $\hone$ atom supported in a ball $B = B(x_0,r)$, and has vanishing integral when $r < 1$.
Then 
$$
\|\cMb a\|_{L^1} \lesssim \|a(b-c_{B})\|_{\hone} + \|b\|_{\bmo} \lesssim \|\cMb a\|_{L^1}+ \|b\|_{\bmo}.
$$
Moreover,
\begin{align*}
\bigg|\int ab\bigg|\lesssim \min\bigg\lb\frac{\|\cMb a\|_{L^1}+ \|b\|_{\bmo}}{\log(1+r^{-1})},  \|b\|_{\bmo}\bigg\rb.
\end{align*} 
\end{proposition}

\begin{proof}
We assume that $a$ satisfies an $L^2$ size condition.
First note that by \eqref{eqn:abc}, $a(b-c_{B}) \in L^s(\Rn)$ for $1 < s < 2$, and it has compact support,  $a(b-c_{B})$ is in $\hone(\Rn)$.
By Definition~\ref{def:cmbf}, we have, for almost every $x \in \Rn$,
$$\cMb(a)(x) = \cM([b(x) - b]a)(x) \leq |b(x) - c_B|\cM(a)(x) + \cM([b - c_B]a)(x),$$
hence
\begin{align*}
 \|\cMb(a)\|_\Lone
& \leq \int |b(x) - c_B|\cM(a)(x)dx + \|\cM([b - c_B]a)\|_\Lone\\
& \lesssim \int |b(x) - c_B|\cM(a)(x)dx +  \|a(b-c_{B})\|_{\hone}.
\end{align*} 
Conversely,
$$\cM([b - c_B]a)(x) \leq \cMb(a)(x) + |b(x) - c_B|\cM(a)(x),$$
so
\begin{align*}
\|a(b-c_{B})\|_{\hone} & \leq \|\cMb(a)\|_\Lone + \int |b(x) - c_B|\cM(a)(x)dx.
\end{align*} 
It thus remains to show that the integral of $|b - c_B|\cM(a)$ is controlled by $\|b\|_\bmo$. The arguments are similar to those in the proof of Proposition~\ref{h1binc}.  First, we have
$$\int_{2B} |b - c_B| \cM(a) 
\leq \|b - c_B\|_{L^2(2B)} \|M a\|_{L^2(2B)} \lesssim  \|b\|_\bmo |B|^{\frac12} \|a\|_{L^2} \lesssim \|b\|_\bmo.$$
Moreover, \eqref{eqn:intersection} implies that if $r \geq 1$, $\cM(a)$ vanishes outside $2B$.  For $r < 1$, we can use the cancellation condition on $a$ to write,
 as in \eqref{eqn:not2B}, for every test function $\phi$, and for almost every $x \in (2B)^c$, 
\begin{align*}
|b(x)-c_B|\bigg|\int a(y) \phi(y)dy\bigg| 
& \leq  |b(x)-b_B|\|\grad \phi\|_\Linfty r \int|a(y)|dy \\
&\lesssim \frac{|b(x)-b_{B}|r}{|x-x_0|^{n+1}}.
\end{align*}
The integral on $(2B)^c$ is then estimated as in the first term of \eqref{eqn:not2Bint}.

By applying Proposition~\ref{prop:meanestimate} to $a(b-c_{B})\in h^1(\Rn)$, followed by the estimates above, we have 
\begin{align*}
\bigg|\int ab\bigg| = \bigg|\int a(b-c_{B}) \bigg| \lesssim \frac{\|a(b-c_{B})\|_{h^1}}{\log(1+r^{-1})}\lesssim \frac{(\|\cMb a\|_{L^1}+ \|b\|_{\bmo})}{\log(1+r^{-1})}.
 \end{align*}
The other upper bound follows from the duality of $\bmo$ and $\hone$, or by Cauchy-Schwartz.
\end{proof}

It is interesting to note that if we had not assumed exact cancellation on $a$ in the hypotheses of Proposition~\ref{ainh1bandh1} when $r < 1$, we would have ended up having to estimate an extra term of the form
$$
\int_{2r<|x-x_0|<2}  \frac{|b(x) - b_B|}{|x-x_0|^{n}}dx \bigg| \int a \bigg|.
$$
Looking at the proof of \eqref{eqn:bintegral}, if $\delta = 0$ and $p = 1$, the integral in $x$ can be estimated by $\|b\|_\bmo(\log \frac 1 r)^2$.  Therefore, we would need a cancellation condition of the form  $\big|\int a\big|\leq[\log(1+r^{-1})]^{-2}$, as in condition (3) of Definition~\ref{def:Approxh1bAtom}.

Combining Propositions~\ref{prop:decomoph1b}, \ref{h1binc} and \ref{ainh1bandh1}, we get the following. 
\begin{corollary}
Let $f\in \hone(\Rn)$. Then the following are equivalent.
\begin{enumerate}
\item The function $f\in \honeatomb(\Rn)$.
\item There exists a sequence $\lb \lambda_j\rb\in \ell^1$ and a collection of $h^1$ atoms $\lb a_j\rb$, where $B_j$ denotes the ball containing the support of $a_j$ and $a_j$ has vanishing integral  when $r(B_j)< 1$, so that $f=\sum_{j} \lambda_j a_j$ and $\cMb a_j$ are uniformly bounded in $L^1(\Rn)$.
\item There exists a sequence $\lb \lambda_j\rb\in \ell^1$ and a collection of $h^1$ atoms $\lb a_j\rb$, where $B_j$ denotes the ball containing the support of $a_j$ and $a_j$ has vanishing integral  when $r(B_j)< 1$, so that $f=\sum_{j} \lambda_j a_j$ and
$$\sum_{j}\lambda_j a_j(b-c_{B_j}) \mbox{ converges absolutely in } \hone(\Rn).$$
\end{enumerate}
\end{corollary}

\subsection{The case of $b$ in lmo}
We now come to the promised result which gives us the second equality in \eqref{eqn:bconstant} even when $b$ is not constant. Recall that  the pointwise multipliers of $\bmo(\Rn)$,  and hence $h^1(\Rn)$, were identified as the elements of $L^{\infty}\cap \lmo(\Rn)$ in \cite{BF}.  It is therefore not surprising that we have the following results.

\begin{theorem} 
\label{h1atombish1}
Let $b\in \lmo(\Rn)$. Then $\honeatomb(\Rn)=h^1(\Rn)$. Moreover,
\begin{align*}
\|f\|_{h^1} \leq \|f\|_{\honeatomb} \lesssim \frac{\|b\|_{\LMOloc}}{\|b\|_\BMOloc}\|f\|_{h^1}.
\end{align*}
\end{theorem}

\begin{proof}
It was already observed following Definition~\ref{def:Approxh1bAtom} that for $b \in \bmo$, $\honeb$ atoms are approximate $\hone$ atoms (up to a factor of $\log 2$) and therefore $\honeb(\Rn) \subset \hone(\Rn)$ with $\|f\|_{h^1} \lesssim \|f\|_{\honeatomb}$.  We only need the $\lmo$ condition for the reverse inclusion.

Let $a$ be an $\hone$ atom with support in the ball $B = B(x_0,r)$, with $L^2$ size condition and such that  $\int a =0$. This means that conditions (1), (2), and the first part of (3) in Definition~\ref{def:Approxh1bAtom} are satisfied, and we only need to check the approximate cancellation against $b$.  Since this holds automatically for $r\geq 1$, we assume $r < 1$.
By the $L^2$ size condition on $a$, 
\begin{align*}
\bigg|\int_{\Rn} a(x)[b(x)-b_B]dx \bigg|
&\leq\frac{1}{\log(1+r^{-1})} \frac{\log(1+r^{-1})}{|B|^{\frac12}}\bigg(\int_B |b(x)-b_B|^2 dx\bigg)^{\frac{1}{2}} \\
&\leq \frac{\|b\|_{\LMO_{\text{loc},2}}}{\log(1+r^{-1})},
\end{align*}
where 
$$\|b\|_{\LMO_{\text{loc},2}} := \sup_{r(B)<1}\log(1+r(B)^{-1})\bigg(\fint_B |b(x)-b_B|^2 dx\bigg)^{\frac{1}{2}}.$$
Since 
$$\gamma := \frac{\log 2 \|b\|_{\BMO_{\text{loc},2}}}{\|b\|_{\LMO_{\text{loc},2}}} \leq 1$$
and
$\log 2 \|b\|_{\BMO_{\text{loc},2}} \leq \|b\|_{\bmo,2} = C_b$,
we have that  $\gamma a$ is an $\honeb$ atom. 

Finally, for $f\in \hone$ with an atomic decomposition $f=\sum_{j=1}^{\infty} \lambda_j a_j$, we can regard it as a decomposition using $\honeb$ atoms multiplied by $\gamma^{-1}$, so $f \in \honeatomb(\Rn)$ with 
$$\|f\|_\honeatomb \leq \gamma^{-1} \sum_{j=1}^{\infty} |\lambda_j|.$$
Taking the infimum over all such decompositions, we get 
$\|f\|_{\honeatomb} \leq \gamma^{-1} \|f\|_{\hone}.$
The inequality in the statement of the theorem is obtained by applying the John-Nirenberg inequality for $\LMO$ - see \cite{ShiTorchinsky}.
\end{proof}

We give a partial converse to the theorem.
\begin{proposition} \label{lmoequal1}
Let $b\in \Loneloc(\Rn)$. Suppose there exists a constant $C_b$ such that every $\hone$ atom $a$  satisfies 
$$\bigg|\int ab\bigg|\leq \min\bigg\lb C_b,\frac{C_b\log(2)}{\log(1+r^{-1})}\bigg\rb,$$ 
where $r$ is the radius of the ball $B$ containing the support of $a$. Then $b\in \lmo(\Rn)$.
\end{proposition}

\begin{proof}
We first handle a ball $B$ with radius $r<1$. We define $s(x):= \sgn[b(x)-b_B]$ and $a(x):= \frac{s(x)-s_B}{|B|}\chi_B(x)$. Then $a$ is an $h^1$ atom. Moreover, we can write 
\begin{align*}
\log(1+r^{-1})\fint_{B} |b(x)-b_B|dx 
&=\log(1+r^{-1})\fint_{B} s(x)(b(x)-b_B)dx \\
&=\frac{\log(1+r^{-1})}{|B|} \int_{B} (s(x)-s_B)(b(x)-b_B)dx \\
&=\log(1+r^{-1})\int_{B} a(x)b(x)dx\\
&=\log(1+r^{-1}) \bigg|\int_{B} ab\bigg|\\
&\leq C_b.
\end{align*}

If $r\geq 1$,  we define $a(x):= \frac{\sgn[b(x)]}{|B|}\chi_B(x)$. A similar argument shows that
\begin{align*}
\frac{1}{|B|}\int_{B} |b(x)|dx &= \bigg|\int_{B} a(x)b(x)dx\bigg| \leq C_b.
\end{align*}
Therefore, we can conclude that $b \in \lmo(\Rn)$ with $\|b\|_\lmo \leq C_b$.
\end{proof}

In a similar vein, we give a partial converse to Proposition~\ref{h1binc}
\begin{proposition}
 \label{lmoequal2}
Let $b\in \bmo(\Rn)$.  Suppose there exists a constant $\kappa_b$ such that every $\hone$ atom $a$  satisfies 
$$\|\cMb a\|_\Lone \leq \kappa_b.$$  
Then $b$ satisfies 
\begin{align*}
\sup_{r(B)<1}\frac{1}{|B|}\||b-b_B|\chi_B\|_{h^1}+\sup_{r(B)\geq1}\frac{1}{|B|}\||b|\chi_B\|_{h^1}<\infty.
\end{align*}
\end{proposition}

\begin{proof}
Fix a ball $B$ with radius $r<1$. As in the previous proof, set $s(x):= \sgn[b(x)-b_B]$ and $a(x):= \frac{s(x)-s_B}{|B|}\chi_B(x)$.
Then $a$ is an $\hone$ atom, and from Proposition~\ref{ainh1bandh1},
\begin{align*}
\||b-b_B|\frac{\chi_B}{|B|}\|_{h^1} &=\|(b-b_B)s(x)\frac{\chi_B}{|B|} \|_{h^1} \\
&\leq \|a(b-b_B)\|_{h^1}+|s_B|\bigg\|\frac{(b-b_B)\chi_B}{|B|}\bigg\|_{h^1}   \\
&\lesssim \|\cM_b(a)\|_{L^1}+\|b\|_{\bmo} + |s_B|\bigg\|\frac{(b-b_B)\chi_B}{|B|}\bigg\|_{h^1} \\
&\lesssim \|\cM_b(a)\|_{L^1}+\|b\|_{\bmo}.
\end{align*}
In the last step we used the fact that  $|s_B| \leq 1$ and $\frac{(b-b_B)\chi_B}{|B|\|b\|_{\bmo,2}}$ is a $(1,2)$ atom.

For $B$ with radius $r\geq 1$, we again put $a(x)=\frac{\sgn(b(x))\chi_B(x)}{|B|}$ and use Proposition~\ref{ainh1bandh1}:
\begin{align*}
\||b|\frac{\chi_B}{|B|}\|_{h^1} = \|a(b - c_B)\|_{h^1} 
\lesssim \|\cM_b(a)\|_{L^1} + \|b\|_{\bmo}.
\end{align*}
\end{proof}

One interesting corollary from Proposition \ref{h1binc} and \ref{lmoequal2} is the following.

\begin{corollary} \label{lmoh1bmo}
Let $b\in \bmo(\Rn)$. Then the following are equivalent.
\begin{enumerate}
\item The function $b$ is in $\lmo(\Rn)$.
\item The function $b$ satisfies 
\begin{align*}
A_b: = \sup_{r(B)<1}\frac{1}{|B|}\||b-b_B|\chi_B\|_{h^1}+\sup_{r(B)\geq1}\frac{1}{|B|}\||b|\chi_B\|_{h^1}<\infty.
\end{align*}
\end{enumerate}
\end{corollary}

\begin{proof}
Suppose $b\in \lmo(\Rn)$. Then from the proof of Theorem~\ref{h1atombish1}, for  an $h^1$ atom $a$,  $\gamma a$ is an $h^1_b$ atom, hence by Proposition~\ref{h1binc} 
$$\|\cMb a\|_{\Lone}\lesssim \gamma^{-1} \|b\|_\bmo \approx \kappa_b := \frac{\|b\|_{\LMOloc}}{\|b\|_\BMOloc}\|b\|_\bmo,$$
which is exactly the hypothesis of Proposition~\ref{lmoequal2}.

Conversely, suppose $A_b < \infty$. Observe that from Proposition~\ref{prop:meanestimate}, we can estimate
\begin{align*}
\sup_{r(B)<1} \frac{\log(1+[r(B)]^{-1})}{|B|}\int_B|b(x)-b_B|dx &\lesssim  \sup_{r(B)<1} \frac{1}{|B|}\||b-b_B|\chi_B\|_{h^1}.
\end{align*}
Meanwhile, using the fact that $\|f\|_{L^1(\Rn)}\leq \|f\|_{h^1(\Rn)}$, we have
\begin{align*}
\sup_{r(B)\geq 1} \frac{1}{|B|}\int_{B}|b(x)|dx \leq \sup_{r(B)\geq1}\frac{1}{|B|}\||b|\chi_B\|_{h^1}.
\end{align*}
Thus $\|b\|_\lmo \lesssim A_b < \infty$.
\end{proof}

\section{Boundedness of the commutator of inhomogeneous singular integral operators with bmo functions}
\label{sec:commutators}
At the end of Section~\ref{sec:LCK}, we introduced a class of Calder\'on-Zygmund operators where the kernel has better decay away from the diagonal (see Definition~\ref{def:inhom}).  We will now show that the commutators of these operators with $\bmo$ functions take $\honeb$ atoms to $\Lone$.

\begin{theorem} 
\label{h1bL1}
Suppose $b\in \bmo(\Rn)$ and $T$ is a inhomogeneous singular integral operator.  Then $[b,T](a)\in L^1(\Rn)$ and 
$$\|[b,T]a\|_{L^1}\lesssim\|b\|_{\bmo}$$ 
for all $\honeb$ atoms $a$. Thus $[b,T]:\honefinb(\Rn) \mapsto L^1(\Rn)$, where $\honefinb(\Rn)$ denotes the space of finite linear combinations of $\honeb$ atoms.
\end{theorem}

\begin{proof}
Let $a$ be an $\honeb$ atom with $\supp(a)\subset B=B(x_0,r)$. 
Note that we can write the commutator acting on $a$ as 
$[b,T](a)= (b-c_B)T(a)-T(a(b-c_B))$. 

We first show that $T(a(b-c_B))\in L^1(\Rn)$.  By the $L^p$ boundedness of $T$ for $p = \frac 3 2$, we have
\begin{align*}
\|T(a(b-c_B))\|_{L^1(2B)} &\leq |2B|^{\frac13} \|T(a(b-c_B))\|_{L^{\frac32}(2B)} \\
&\lesssim |B|^{\frac13} \|T\|_{L^{\frac32}\mapsto L^{\frac32}} \|a(b-c_B)\|_{L^{\frac32}(\Rn)}\\
&\lesssim  \|T\|_{L^{\frac32}\mapsto L^{\frac32}}|B|^{\frac13}\|b-c_B\|_{L^6(B)}\|a\|_{L^2(B)}\\
&\lesssim  \|T\|_{L^{\frac32}\mapsto L^{\frac32}}|B|^{-\frac16}\|b-c_B\|_{L^6(B)}\lesssim  \|T\|_{L^{\frac32}\mapsto L^{\frac32}}\|b\|_{\bmo,6}.
\end{align*}

Looking at $x \in (2B)^c$,  first consider the case $r<1$.  From Proposition~\ref{prop:decomoph1b}, we may assume $\int a = 0$.  Since $x \not\in \supp(a(b-b_B))$, we can use the kernel representation to write
\begin{align*}
&\|T(a(b-b_B))\|_{L^1((2B)^c)}\\
&\leq \int_{(2B)^c} \bigg|\int_{B}K(x,y)a(y)(b(y)-b_B)dy \bigg|dx \\
&\leq \int_{(2B)^c} \bigg|\int_{B}[K(x,y)-K(x,x_0)]a(y)(b(y)-b_B)dy \bigg|dx 
+ \int_{(2B)^c} |K(x,x_0)| \bigg|\int ab \bigg|dx\\
&=: I +II.
\end{align*}
As $2|y - x_0| < 2r \leq |x - x_0|$, we can use the smoothness of $K$ (condition (2) in Definition~\ref{def:inhom}) to estimate
\begin{align*}
I&\lesssim \int_{(2B)^c}\int_{B} \frac{|y-x_0|^{\delta}}{|x-x_0|^{n+\delta}} |a(y)||b(y)-b_B|dydx \\
&\lesssim \int_{(2B)^c}\frac{r^{\delta}}{|x-x_0|^{n+\delta}}\int_{B} |a(y)||b(y)-b_B|dy dx \\
&\leq \int_{(2B)^c}\frac{Cr^{\delta}}{|x-x_0|^{n+\delta}} dx \|b\|_{\bmo,2}\\
&\lesssim\|b\|_{\bmo}.
\end{align*}
Using the decay property of $K$ (condition (2) in Definition~\ref{def:inhom}) and the cancellation of $a$ against $b$, recalling that we chose $C_b =  \|b\|_{\bmo,2}$, we have
\begin{align*}
II &\leq  \int_{(2B)^c} |K(x,x_0)| \bigg|\int ab \bigg|dx \\
&\lesssim \bigg(\int_{2r\leq |x-x_0|\leq 1}\frac{1}{|x-x_0|^n}dx + \int_{|x-x_0|\geq 1}\frac{1}{|x-x_0|^{n+\varepsilon}}dx\bigg)\frac{C_b}{\ln(1+r^{-1})} \\
&\lesssim  \|b\|_{\bmo,2}\frac{\ln(1+r^{-1})+1}{\ln(1+r^{-1})} \lesssim  \|b\|_{\bmo}.
\end{align*}

When $r\geq 1$, using the fact that $2|y-x_0|\leq |x-x_0|$ implies $|x-y|\geq \frac{1}{2}|x-x_0|$, the decay of $K$ again gives
\begin{align*}
\|T(a(b-c_B))\|_{L^1((2B)^c)}&\leq \int_{(2B)^c} \bigg|\int_{B}K(x,y)a(y)b(y)dy \bigg|dx \\
&\lesssim \int_{|x-x_0|\geq 2r} \frac{1}{|x-x_0|^{n+\varepsilon}}  \|b\|_{\bmo,2}\\
&\lesssim \int_{|x-x_0|\geq 2} \frac{1}{|x-x_0|^{n+\varepsilon}}  \|b\|_{\bmo,2} \lesssim \|b\|_{\bmo}.
\end{align*}

Now we can handle the term $(b-c_B)T(a)$. Using the boundedness of $T$ on $L^2$, we have 
\begin{align*}
\|(b-c_B)T(a)\|_{L^1(2B)} &\leq \|b-c_B\|_{L^2(2B)}\|T(a)\|_{L^2(\Rn)}\\
&\lesssim \|b\|_{\bmo,2} |B|^{\frac12} \|T\|_{L^2\mapsto L^2} \|a\|_{L^2(\Rn)} \\
&\lesssim \|b\|_{\bmo}\|T\|_{L^2\mapsto L^2}.
\end{align*}
When $r<1$, we apply the cancellation of $a$, the smoothness of $K$ and \eqref{eqn:bintegral} to get
\begin{align}
\nonumber
\|(b-b_B)T(a)\|_{L^1((2B)^c)} &\leq \int_{(2B)^c} |b(x)-b_B| \bigg|\int_{B}[K(x,y)-K(x,x_0)]a(y)dy\bigg|dx \\
\label{eqn:bbBTa}
&\leq  \int_{(2B)^c} |b(x)-b_B|\frac{r^{\delta}}{|x-x_0|^{n+\delta}}dx \lesssim\|b\|_{\bmo}.
\end{align}
When $r\geq 1$, similar to the estimate of $\|T(a(b-c_B))\|_{L^1((2B)^c)}$, we have, again by  \eqref{eqn:bintegral},
\begin{align}
\nonumber
\|(b-c_B)T(a)\|_{L^1((2B)^c)} 
\nonumber
&\leq \int_{(2B)^c} |b(x)| \bigg|\int_{B}K(x,y)a(y)dy\bigg|dx \\
\nonumber
&\leq \int_{(2B)^c} \frac{|b(x)|}{|x-x_0|^{n+\varepsilon}} dx \\
\label{eqn:rgeq1}
&\leq \int_{(2B)^c} r^{\varepsilon}\frac{|b(x) - c_B|}{|x-x_0|^{n+\varepsilon}} dx \lesssim \|b\|_{\bmo}.
\end{align}

We have thus shown $\|[b,T]a\|_{L^1}\lesssim\|b\|_{\bmo}$ for all $\honeb$ atoms $a$. Finally, $f \in \honefinb(\Rn)$ means it can be written as a finite linear combination $\sum_{k} \gamma_k a_k$ of $\honeb$ atoms, so by linearity and taking the infimum over all such decompositions, we have
$$\|[b,T]f\| \lesssim \|b\|_{\bmo} \inf \sum_{k} |\gamma_k|.$$
\end{proof}

We cannot extend the conclusion to all of $\honeatomb(\Rn)$ since we do not know whether for $f \in \honefinb(\Rn)$, the infimum $\inf \sum_{k} |\gamma_k|$ taken over all finite linear combinations of atoms is comparable to $\|f\|_\honeatomb$, the infimum taken over all possible atomic decomposition.

When $\honeatomb(\Rn) = \hone(\Rn)$, it is true that the infimum over the finite linear combinations of atoms is comparable to the $\hone$ norm, and it suffices to verify the boundedness on one $\hone$ atom in order to get a bounded extension to the whole space (see \cite[Theorem 3.1, Corollary 3.4]{MSV1} for the case of $\Hone$ and \cite[Proposition 7.1]{HYY} for the case of $\hone$ on a space of homogeneous type).  This gives us the following corollary.

\begin{corollary} \label{h1L1lmo1}
Let $b\in \lmo(\Rn)$, and $T$ be a inhomogeneous singular integral operator. Then there exists a unique extension $L_b$ of $[b,T]$ that is bounded from $\hone(\Rn)$ to $L^1(\Rn)$. 
\end{corollary}

In particular, for $b\in \lmo(\Rn)$, we can apply this corollary to show the boundedness of the commutators $[b, \cRjeta]$ and $[b, \cRjpsi]$ defined in Corollary~\ref{localriesz}, provided the resulting localized convolution kernels have extra decay away from the origin.  This requires some extra assumptions on $\eta$ and $\psi$ in addition to those in Theorem~\ref{Mainthm1}. 

We can also compare this corollary with the results of \cite{HungKy}.  While our assumption $b \in \lmo$ is stronger than their assumption $b \in \LMO_\infty$, which allows the oscillation to grow on large balls, and they obtain $\hone$ to $\hone$ boundedness, the operators they consider, which are pseudo-differential operators, have kernels which have much better decay and smoothness (see \cite[Proposition 2.1]{HungKy}).

In order to get $\hone$ to $\hone$ boundedness of $[b,T]$, we need to impose some approximate cancellation conditions.  
We follow the ideas in \cite[Definition 5.1]{DLPV1}.
\begin{definition} 
\label{T*def}
Suppose $T$ is an inhomogeneous singular integral operator.  Given a ball $B$,  denote by $L^{2}_{0}(\Rn)$ those $g \in L^2(B)$  with $\int g = 0$.  For $b\in \bmo$, define $T^{*}_B(b)$, relative to this ball $B$, in the distributional sense, by
		\begin{equation} 
		\label{T*}
			\langle T^{*}_B(b), g \rangle :=  \int_{\Rn}(b - b_B)Tg(x)dx, \quad \quad \forall\, g \,\in L^{2}_{0}(B). 
		\end{equation} 
\end{definition}
As in \eqref{eqn:bbBTa}, the conditions on the kernel and the vanishing integral of $g$ guarantee that the integral on the right-hand-side of \eqref{T*} converges absolutely and $T^{*}_B(b)$ is a bounded linear functional on $L^{2}_{0}(B)$ with norm bounded by a constant multiple of $\|b\|_\bmo|B|^{1/2}$.  Thus we can identify $T^{*}_B(b)$ with an equivalence class of functions in $L^2(B)$ modulo constants (similarly to the case $p = 1$ in \cite[Remark 5.4]{DLPV1}).  In particular, we can impose the following condition, denoting $f:= T^{*}_B(b)$, without ambiguity:
\begin{equation}
\label{T*b}
\left(\fint_{B}|f - f_B|^2\right)^{1/2}\leq \log(1 + r(B)^{-1}).
\end{equation}
 Similarly to what was pointed out in \cite[Remark 5.4]{DLPV1} for the case of $T^*((\cdot - x_B)^\alpha)$, this condition is not the same as requiring $T^*(b) \in \LMOloc(\Rn)$ since $T^{*}_B(b)$ is defined relative to a specific ball $B$, using the $L^2$ pairing, not the pairing of $\bmo$ and $\hone$, and we are imposing condition \eqref{T*b} only for that ball.
 
While the local Campanato-type conditions used in \cite[Theorem 5.3]{DLPV1} to guarantee that  inhomogeneous singular integral operators are bounded on $h^p(\Rn)$ are also expressed in terms of individual balls,  in the case $p = 1$ we do not need to subtract any constant as the power of the monomial is $\alpha = 0$.  Thus the condition becomes $T^*1 \in \LMOloc(\Rn)$, where $T^*$ can be taken in the sense of duality of $\bmo$ and $\hone$.  In general, for $b$ bounded, the pairing between $\bmo$ and $\hone$ can be given by integration, but for arbitrary $b \in \bmo$,  we need to subtract $b_B$ in the integral in \eqref{T*}, unless we have $\int Tg = 0$.  This is the case when $T$ is a convolution operator, as $Tg$ inherits the cancellation from $g$ (see also the remarks following Corollary~\ref{localriesz}).

\begin{theorem} 
\label{thm: h1bh1}
Let $b\in \bmo(\Rn)$. Suppose $T$ is an inhomogeneous singular integral operator. If  $T^*(1) \in \LMOloc(\Rn)$, and, for any ball $B$ with $r(B)<1$, $f:= T^{*}_B(b)$ satisfies \eqref{T*b}, then
$$\|[b,T](a)\|_{\hone}\leq C(1+\|b\|_{\bmo})$$
for all $\honeb$ atoms $a$.
\end{theorem}

\begin{proof}
Let $a$ be an $\honeb$ atom supported in a ball $B=B(x_0,r)$; as above, we may assume that when $r < 1$, $\int a = 0$.
Again write
$[b,T](a)= (b-c_B)T(a)-T(a(b-c_B))$. 

As explained before the statement of the theorem, the assumption $T^*(1)\in \LMOloc(\Rn)$ allows us to apply \cite[Theorem 5.3]{DLPV1}  get  the boundedness of $T$ on $\hone$; thus from
 \eqref{eqn:honeb-abc} we get
$$\|T(a(b-c_B))\|_{\hone}\lesssim \|a(b-c_B)\|_{\hone} \lesssim \|b\|_\bmo.$$ 

It remains to estimate the term $\|(b-c_{B})T(a)\|_{\hone}$.  We will show that $M = (b-c_{B})T(a)$ is a multiple of an $(s,\lambda)$ molecule (see Definition~\ref{def:molecule}) associated to the ball $2B$, where we choose 
$$s=\frac32,\quad \lambda= \frac{n}{2}+ \mu, \quad 0 < \mu < \frac 3 2 \min(\delta,\varepsilon).$$

We first verify condition M1. Proceeding as in  \eqref{eqn:abc} and using the boundedness of $T$ on $L^{2}$, we have
\begin{align*}
\|(b-b_{B})T(a)\|_{L^{\frac{3}{2}}(2B)} & \lesssim\|(b-b_{B})\|_{L^6(2B)} \|T(a)\|_{L^2}\\
& \lesssim \|b-b_{B}\|_{L^6(2B)}\|a\|_{L^2}\\
&\lesssim \|b\|_{\bmo,6} r^{-\frac{n}{3}}.
\end{align*}
Next, we need to show condition M2:
$\ds \int_{(2B)^c}|M(x)|^{\frac32}|x-x_0|^{\lambda}dx \lesssim r^{\mu}$.  Suppose that $r < 1$.
As in \eqref{eqn:bbBTa}, using the cancellation on $a$ and the smoothness of the kernel to bound $T(a)$, and \eqref{eqn:bintegral}, we have 
\begin{align*}
&\int_{(2B)^c}|(b(x)-b_{B})T(a)(x)|^{\frac32}|x-x_0|^{\lambda}dx \\
&\lesssim \int_{(2B)^c} |b(x)-b_{B}|^{\frac32} \frac{r^{\frac{3}{2}\delta}}{|x-x_0|^{\frac{3n}{2}+\frac{3\delta}{2}} }|x-x_0|^{\lambda}dx 
 =  r^{\frac32 \delta} \int_{(2B)^c}\frac{ |b(x)-b_{B}|^{\frac32}}{|x-x_0|^{n+\frac{3}{2}\delta-\mu}}dx \\
&\lesssim \|b\|_{\bmo}^{\frac32} r^{\mu}.
\end{align*}
Finally, condition \eqref{T*b} on $f = T_B^*(b)$ gives us
\begin{align*}
\bigg|\int_{\Rn} [b(x)-b_{B}]T(a)(x)dx\bigg|  & = \bigg|\langle T^{*}_B(b), a \rangle\bigg| 
 \leq \left(\int_{B}|f - f_B|^2\right)^{1/2}\|a\|_{L^2(B)} \\
& \leq \log(1 + r(B)^{-1}).
\end{align*}

Now looking at the case $r \geq 1$, we can proceed as in \eqref{eqn:rgeq1} to write
\begin{align*}
&\int_{(2B)^c}|(b(x)-c_{B})T(a)(x)|^{\frac32}|x-x_0|^{\lambda}dx \\
&\lesssim \int_{(2B)^c} \frac{|b(x)-c_{B}|^{\frac32}}{|x-x_0|^{{\frac32}n+{\frac32}\varepsilon}} |x-x_0|^{\lambda}dx \\
&\leq r^{\frac{3}{2}\varepsilon}  \int_{(2B)^c}\frac{ |b(x)-c_{B}|^{\frac32}}{|x-x_0|^{n+\frac{3}{2}\varepsilon-\mu}}dx \\
&\lesssim \|b\|_\bmo^{\frac32} r^{\mu}.
\end{align*}
Here we have used \eqref{eqn:bintegral} as well as the fact that $r \geq 1$, $\mu < \frac32\varepsilon$.
\end{proof}

Similarly to Corollary \ref{h1L1lmo1}, when $\honeb = \hone$, the boundedness extends from one atom to all functions in the space.

\begin{corollary}
Assuming the hypotheses of Theorem~\ref{thm: h1bh1}, if in addition $b\in \lmo(\Rn)$, then $[b,T]$ is bounded from $\hone(\Rn)$ to $\hone(\Rn)$. 
\end{corollary}

\bigskip
  \footnotesize

 G. Dafni, \textsc{Concordia University, Department of Mathematics and Statistics, Montr\'{e}al, Qu\'{e}bec, H3G-1M8, Canada}\par\nopagebreak

  \medskip

  C. H. Lau, \textsc{Concordia University, Department of Mathematics and Statistics, Montr\'{e}al, Qu\'{e}bec, H3G-1M8, Canada}\par\nopagebreak

  \medskip


\end{document}